\documentclass[11pt, a4paper]{amsart}

\usepackage[utf8]{inputenc}

\usepackage{xcolor}
\definecolor{darkgreen}{rgb}{0,0.5,0}
\usepackage[
        colorlinks, citecolor=darkgreen,
        backref,
        pdfauthor={S. Gajovi\'c, J.S. M\"uller},
        pdftitle={Linear quadratic Chabauty},
]{hyperref}
\usepackage[hyphenbreaks]{breakurl}
\usepackage{fancyhdr}
\usepackage{tikz,enumerate,caption,stmaryrd,amsfonts,amssymb,systeme,comment}
\usetikzlibrary{matrix,positioning,arrows,shapes,chains,calc,automata}
\usetikzlibrary{decorations.pathreplacing}
\usepackage{amssymb}
\usepackage{amsmath}
\usepackage{mathtools}
\usepackage{amsthm}
\usepackage{tikz-cd}
\usepackage[OT2,T1]{fontenc}
\usepackage{url}
\DeclareSymbolFont{cyrletters}{OT2}{wncyr}{m}{n}
\DeclareMathSymbol{\Sha}{\mathalpha}{cyrletters}{"58}

\usepackage{cleveref}
\crefformat{section}{§#2#1#3}
\usepackage{colonequals}

\newcommand{\Q}{\mathbb{Q}}

\newcommand{\calU}{\mathcal{U}}

\newcommand{\calD}{\mathcal{D}}
\newcommand{\Z}{\mathbb{Z}}

\newcommand{\F}{\mathbb{F}}

\renewcommand{\O}{\mathcal{O}}
\newcommand{\p}{\mathfrak{p}}
\newcommand{\q}{\mathfrak{q}}
\newcommand{\hdr}{\operatorname{H^1_{dR}}}

\newcommand{\Div}{\operatorname{Div}}
\newcommand{\dv}{\operatorname{div}}

\newcommand{\Res}{\operatorname{Res}}

\newcommand{\tr}{\operatorname{tr}}
\newcommand{\id}{\operatorname{id}}
\newcommand{\Spec}{\operatorname{Spec}}

\newcommand{\supp}{\operatorname{supp}}
\newcommand{\ord}{\operatorname{ord}}
\newcommand{\rk}{\operatorname{rk}}

\renewcommand{\sp}{\operatorname{sp}}
\renewcommand{\O}{\mathcal{O}}

\newtheorem{theorem}{Theorem}[section]
\newtheorem{proposition}[theorem]{Proposition}

\newtheorem{lemma}[theorem]{Lemma}
\newtheorem{cond}{Condition}

\newtheorem{alg}[theorem]{Algorithm}
\theoremstyle{remark}
\newtheorem{remark}[theorem]{Remark}
\newtheorem{definition}[theorem]{Definition}

\numberwithin{equation}{section}

\title{Linear quadratic Chabauty}
\author[Stevan Gajovi\'c]{Stevan Gajovi\'c}
\address{Stevan Gajovi\'c, Charles University, Faculty of Mathematics and Physics, Department of Algebra, Sokolov\-sk\' a 83, 186~75 Praha~8, Czech Republic}
\email{gajovic@karlin.mff.cuni.cz}

\author[J.~Steffen M\"uller]{J.~Steffen M\"uller}
\address{J.~Steffen M\"uller, University of Groningen, Faculty of Science and Engineering, Algebra, Nijenborgh~9, 9747~AG Groningen, Netherlands}
\email{steffen.muller@rug.nl}
\date{}

\begin{document}

\maketitle

\begin{abstract}
We introduce a new quadratic
  Chabauty method to compute the integral points on certain even degree
  hyperelliptic curves. Our approach relies on a nontrivial degree zero divisor supported at the two
  points at infinity to restrict the $p$-adic height to a linear function;
  we can then express this restriction in terms of holomorphic Coleman integrals under the
  standard quadratic Chabauty assumption. Then we use this linear relation
  to extract the integral points on the curve. We also generalize our method to
  integral points over number fields.
  Our method is significantly simpler and faster than all other existing versions of the quadratic Chabauty method. We give examples over $\Q$ and $\Q(\sqrt{7})$.
\end{abstract}

\section{Introduction}\label{sec:intro}
Suppose that $X/\Q$ is a hyperelliptic curve, given by an affine equation  $y^2=f(x)$ with 
$f\in \Z[x]$ squarefree.
We set 
\begin{equation*}\label{D:calU}
  \mathcal{U}=\Spec\Z[x,y]/(y^2-f(x))\,.
\end{equation*}
Balakrishnan, Besser and the
second author developed in~\cite{BBM1} a $p$-adic method, called
{\it quadratic Chabauty}, to
compute the set $\calU(\Z)$ of integral points in $X(\Q)$ in the following
situation: Suppose that $f$ is monic and of odd degree $2g+1>2$,
let $p$ be a prime of good reduction for
$X$,  suppose that the rank $r$ of the Mordell--Weil group $J(\Q)$ is equal to the genus
$g$ of $X$, and that the closure of $J(\Q)$ has finite index in $J(\Q_p)$.
In~\cite{BBM1}, the authors construct a nontrivial locally
analytic function $\rho\colon \calU(\Z_p)\to \Q_p$ such that $\rho(\calU(\Z))\subseteq T$, for a finite and computable set $T\subset \Q_p$.
The construction depends on properties of the
$p$-adic height pairing $h$ introduced by Coleman and
Gross~\cite{Coleman-Gross-Heights}.
This result was later reproved in~\cite{Jen-Netan-QCRP1} by
Balakrishnan--Dogra via Chabauty--Kim theory; using similar ideas, they
also developed a quadratic Chabauty method for rational points on certain
curves (since this assumes that the Picard number is at least~2, it does
not supersede the quadratic Chabauty method for integral points, which has
no such assumption).
It was later extended in \cite{BBBM-QCNF} to the case of $\O_K$-integral points on certain hyperelliptic curves $X$ of odd degree over number fields $K$. 
This quadratic Chabauty method can be used to compute $\calU(\Z)$ explicitly,
see~\cite{BBM1,BBM2}. For this, one can use algorithms for $h$ due to Balakrishnan and
Besser~\cite{BBHeights} and the second author~\cite{Mul14}.
The construction of $\rho$ involves the following tools:
\begin{enumerate}
\item[(i)] $p$-adic Arakelov theory as developed by Besser \cite{Besser-padic-Arakelov};
\item[(ii)] an extension of local Coleman--Gross heights to 
divisors with common support;
\item[(iii)] double Coleman integrals;
\item[(iv)] Coleman integrals where one of the endpoints is a tangential base point.
\end{enumerate}
Therefore this construction is quite technical. In the present paper, we
find a similar but much simpler method to construct a function $\rho$ that
we use to compute integral points on {\it even} degree hyperelliptic curves
(under certain conditions). Our construction does not require any of the
tools (i)--(iv).

\subsection{Linear quadratic Chabauty for $\Z$-integral points}
In Section~\ref{sec:QC-over-Z}
we introduce a new quadratic Chabauty method for integral points on certain
even degree hyperelliptic curves $X\colon y^2=f(x)$ satisfying $r=g$. 
Our approach differs from that of~\cite{BBM1, BBM2} in that it relies on the
existence of a nontrivial $\Q$-rational divisor of degree~0 on $X$
supported at infinity, which exists precisely when the leading coefficient
of $f$ is a square. It is also notably faster, since the slowest part of
the quadratic Chabauty method in~\cite{BBM1, BBM2} is the computation of
the double Coleman integrals, which our approach avoids.

Our method is based on the following result, which we prove
in~\cref{subsec:QC-Z-integral}.
\begin{theorem}\label{T:QC-integrals-pts-Z}
Let $X\colon y^2=f(x)$ be a hyperelliptic curve over $\Q$ such that $f\in
  \Z[x]$ is a squarefree polynomial of degree $2g+2$ whose leading
  coefficient is a square in $\Z$. Suppose $r=g$. Let $p>g$ be a prime of
  good reduction for $X$. 
  Then there is a nontrivial computable locally
  analytic function $\rho\colon \calU(\Z_p)\to \Q_p$ and a finite and computable set $T$, 
  such that $\rho(\calU(\Z))\subseteq T$. 
\end{theorem}
We briefly describe the construction of $\rho$ and $T$.
 They rely on the global $p$-adic height pairing $h$ due
to Coleman and Gross~\cite{Coleman-Gross-Heights}.
This is a sum of local heights \[h=\sum_{v\in
\Z\;\text{prime}}h_v=h_p+\sum_{q\neq p\; \text{finite prime}}h_q\,.\] 
More precisely, we use 
 Coleman integrals and the local height $h_p$ (see
 \cref{subsec:intro-heights-at-p})
to construct $\rho$,  and we analyze local heights away from
 $p$ to obtain the set $T$ (see \cref{subsec:intro-heights-away-p} and
\cref{sec:heights_away}). 

We may assume that $f$ is monic and that the functions
$f_0,\ldots,f_{g-1}\colon J(\Q)\to \Q_p$
defined by 
$f_i(a)\colonequals \int_0^{a}\frac{x^idx}{y}$ induce a basis of
the space of linear maps $J(\Q)\otimes \Q\rightarrow \Q_p$, since otherwise
the method of Chabauty--Coleman, as discussed
in~\cite{Chabauty-Coleman-MCP12}, may be applied.
Hence the linear function 
$\lambda\colon J(\Q)\rightarrow \Q_p$ given by 
$\lambda(a)=h(\infty_--\infty_+,a)$ can be represented as a $\Q_p$-linear combination
$\lambda=\sum_{i=0}^{g-1} \alpha_i f_i$.
For simplicity, we assume that
we know an integral point $Q\in \calU(\Z)$ (see Remark~\ref{r:no-int-point} for an approach when no such point is known). Then 
\[
    \rho(P)\colonequals \sum_{i=0}^{g-1} \alpha_i \int_{Q}^P\frac{x^idx}{y}-h_p(\infty_--\infty_+,P-Q)
\]
defines a computable nontrivial locally analytic function
$\rho\colon \calU(\Z_p)\rightarrow \Q_p$. For $P\in \calU(\Z)$, we have that 
\[
\rho(P)=\sum_{q\neq p}h_q(\infty_--\infty_+,P-Q).
\]
For $q\ne p$ and $P\in \calU(\Z_q)$, Lemma~\ref{L:away}
below shows that $h_q(\infty_--\infty_+,P-Q)$ vanishes for good primes $q$
and takes values in a computable finite set for bad primes $q$, and we
obtain $T$ from these finite sets.

We discuss an explicit method for the
computation of $\calU(\Z)$ stemming from Theorem~\ref{T:QC-integrals-pts-Z}
in \cref{subsec:QC-alg-Z-integral} and we give a precision analysis
in~\cref{subsec:prec-qc-for-Z}.
In \cref{subsec:monic_ex}, we illustrate our method  by 
computing the set of integral points on 
\[
X\colon y^2 = x^6 + 2x^5 - 7x^4 - 18x^3 + 2x^2 + 20x + 9\,.
\]
In this example, the finite subset of $X(\Q_p)$ resulting from our method equals the 
set of integral points on $X$. In general, this equality might not hold,
but in such cases we can combine our approach with the Mordell--Weil sieve
(see~\cite{Bruin-Stoll-MW-Sieve})  to still find the integral points.

Since we assume that the leading coefficient $d_{2g+2}$ of $f$ is a square,
there is also an elementary method to compute the integral points on curves satisfying the
conditions of Theorem~\ref{T:QC-integrals-pts-Z}; namely, a variant of
Runge's method can be applied, see Appendix \ref{appendix:Runge}.
A similar approach works for a negative leading coefficient. 
To deal with non-square positive leading coefficients, one would have to 
extend the method from \cite{BBM1} using $p$-adic Arakelov theory and
double Coleman integrals to this case. 

\subsection{Linear quadratic Chabauty for integral points over
number fields}
Similar to the extension of the techniques of~\cite{BBM1} in \cite{BBBM-QCNF}, 
we can extend our approach to compute the $\O_K$-integral points on certain
even degree hyperelliptic curves $X$ over number fields $K$ of degree
$d>1$. 
This method is based on Theorem~\ref{T:QC-NF-integrals-pts-full-statement},
which generalizes Theorem~\ref{T:QC-integrals-pts-Z}.
For this result, we assume that $\infty_{\pm}\in X(K)$ and that $p$ is a
prime that splits completely in $K$ such that $X$ has good reduction at all primes
above $p$.
As in~\cite{BBBM-QCNF}, a necessary  condition is that
$\rk(J/K)\leq dg-r_K$,  where $r_K = \rk \O_K^\ast$. 
For our extension, we use that the $p$-adic height depends on a choice of
a continuous idèle class character
$\chi\colon \mathbb{A}_K ^{\ast}/K^{\ast}\to
\mathbb{Q}_p$.
While there is a canonical choice for $\chi$ when $K=\Q$, this is false in
general,
and in fact it helps to use the maximal number of independent $\chi$. These give rise to $d-r_K$ equations satisfied by the image
of the $\O_K$-integral points on $X$ in $X(\Q_p)^d$. In addition, 
we use $r_K$ equations coming from restriction of scalars and Coleman
integrals, as in the work of Siksek for Chabauty over number fields \cite{Siksek-Chabauty-NF}. 

When the conditions of Theorem~\ref{T:QC-NF-integrals-pts-full-statement} are satisfied,
then we may use it to compute a subset of 
$X(\Q_p)^d$ which contains the images of all $\O_K$-integral points under
the natural embedding (see
\cref{subsec:alg-OK-points}).
In \cref{subsec:Example-QC-OK-real-points} 
we illustrate our technique by determining the set of $\O_K$-integral
points on 
\[
  y^2 = x^6 + x^5 - wx^4 + (1-w)x^3 + x^2 + wx + w\,,
\]
where $K=\Q(w)$ and $w^2=7$.
This computation required a combination of linear quadratic Chabauty with the
Mordell--Weil sieve. We are not aware of any other method that would be able to handle this example.

In contrast
to the case $X/\Q$, the set computed by our method is not necessarily finite. 
For instance,
when $f\in \Z[x]$ has even degree and its leading coefficient $d_{2g+2}$ is a positive
non-square in $\Z$, then our approach does not allow us to compute the $\O_K$-integral points on $X\colon y^2=f(x)$
for $K=\Q(\sqrt{d_{2g+2}})$, so this approach cannot be used to compute the
$\Z$-integral points either.

\subsection*{Acknowledgements}
It is a pleasure to thank
Jennifer Balakrishnan, Francesca Bianchi and Michael Stoll for many comments and
suggestions, and Amnon Besser,  Bas Edixhoven, Sachi Hashimoto, Enis Kaya, Guido Lido, Samir Siksek and Jaap Top for helpful discussions.
We thank Francesca Bianchi for sharing an implementation of the algorithm
from~\cite[Appendix A]{BBBM-QCNF} and Marius Leonhardt
and Martin L\"udtke for comments on
previous versions of this paper.
We are indebted to Pierre Colmez for suggesting the use of the linear map
$\lambda$, to Jennifer Balakrishnan for sharing this idea with us, and to
Bas Edixhoven for suggesting the title.
We thank the referee for a careful and useful report.
We acknowledge support by DFG through DFG-Grant MU 4110/1-1. 
In addition. S.M. was supported by NWO Grant VI.Vidi.192.106., and S.G. was
supported by a guest postdoc fellowship at the  Max Planck Institute for
Mathematics in Bonn, by Czech Science Foundation GAČR, grant 21-00420M, and
by Charles University Research Centre program UNCE/SCI/022 during various
stages of this project. 

Part of the research in this article has already appeared in the first
author's PhD thesis~\cite{Gajovic-thesis} at the University of Groningen,
written under the second
author's supervision. 

\section{Coleman--Gross $p$-adic heights}\label{sec:heights}
We fix a prime number $p\in \Z$. Let $K$ be a number
field and let $X/K$  be a smooth projective and geometrically integral curve of genus $g>0$, with good reduction at all primes above $p$.
The $p$-adic height on the Jacobian $J/K$ of $X$  depends on the choice of a continuous idèle class character
\begin{equation}\label{idchar}
\chi\colon \mathbb{A}_K ^{\ast}/K^{\ast}\longrightarrow \mathbb{Q}_p
\end{equation}
and of a subspace $W_\p \subset \hdr(X_\p/K_\p)$
complementary to the space of holomorphic forms for each $\p\mid p$. 
Coleman-Gross construct the height pairing on $J$ by defining local height
pairings $h_v$ for all finite primes $v$
of $K$, and then summing them up. We make use of the following
notation: for a finite place $v$ of $K$ we denote by $K_v$ the completion
of $K$ at $v$,  by $\O_v$ its ring of integers and by $\pi_v$ a
uniformizer. We write $X_v\colonequals X\otimes K_v$.

\begin{definition}\label{def:local-height}
  A \textit{local height pairing} (with respect to $\chi$) at a finite prime $v$ of
  $K$ is a continuous, symmetric and bi-additive function $h_v$ that assigns to a pair of
  divisors $D_1,D_2 \in \Div^0(X_v)$ with disjoint support an element of
  $\Q_p$ with the following property: 
for all $f\in K_{v}(X_v)^{\ast}$ and $D_2\in \Div^0(X_v)$ such that $\supp(\dv(f))\cap
  \supp(D_2)=\emptyset$, we have
  $$h_v(\dv(f),D_2)=\chi_{v}(f(D_2)).$$
\end{definition}
The construction of $h_v$ depends on whether $\chi_v$ is unramified 
(see~\cref{subsec:intro-heights-away-p}) or is ramified 
(see~\cref{subsec:intro-heights-at-p}).
When $D_1,D_2\in \Div^0(X)$, the local heights satisfy $h_v(D_1, D_2)\colonequals h_v(D_1\otimes K_v, D_2\otimes K_v)=0$ for all but finitely many $v$. Thus, their sum for $D_1,D_2\in \Div^0(X)$ is well-defined: 
$$
h(D_1,D_2)\colonequals \sum_vh_v(D_1,D_2)\colon \Div^0(X)\times \Div^0(X) \longrightarrow \Q_p\,.
$$
By~\cite[Sections~1 and~5]{Coleman-Gross-Heights}, $h$ descends to a bilinear
pairing on $J(K)$, called the (global Coleman--Gross) {\it $p$-adic height
pairing}:
$$
h \colon J(K)\times J(K)\to \Q_p\,.
$$

\subsection{Continuous $\Q_p$-valued idèle class
characters}\label{subsec:intro-heights-icc} 
In this section, we briefly recall
properties of continuous idèle class characters with codomain $\Q_p$,
following~\cite[\S2.1]{BBBM-QCNF}. See also~\cite[\S6.5.4]{Gajovic-thesis}.

We say that a continuous idèle class character~\eqref{idchar}
is~\textit{ramified} at a place $\q$ if $\chi_\q(\O_{K_\q}^{\ast})\ne 0$.
If $\chi$ is unramified at $\q$, then $\chi_\q$ is completely determined by
$\chi_\q(\pi_\q)$. 
For continuity reasons, $\chi$ is unramified at 
 any place $\q\nmid p$.
On the other hand, for
$\p\mid p$, there is  a $\mathbb{Q}_p$-linear map $\tr_\p\colon K_\p\to \Q_p$
such that 
\begin{equation}\label{ellfact}
  \chi_\p(x) = \tr_\p \circ \log_\p(x)\quad\text{for all}\; x \in
  \O_{K_\p}^{\ast}\,.
\end{equation}
If $\chi$ is ramified at $\p$ (see~\cite[Remark 2.1]{BBBM-QCNF}), there is
a branch of the logarithm $\log_\p\colon K_\p^{\ast} \to K_\p$ such that 
\eqref{ellfact} remains valid on $K_\p^{\ast}$.
From now on, we assume that $\chi$ is ramified at least one $\p\mid p$, and
we fix the branch of the logarithm at $\p\mid p$ to be this $\log_\p$.

By~\cite[Lemma~2.2]{BBBM-QCNF}, $\chi$ is uniquely determined by $\tr_\p$ for
all $\p\mid p$ such that $\chi_\p$ is ramified.
The choice $\tr_\p=\tr_{K_\p/\Q_p}$ in
\eqref{ellfact} for all $\p$ such that $\chi_\p$ is ramified, 
we get the {\it cyclotomic idèle class
character}, see \cite[Lemma 2.3, Example 2.7]{BBBM-QCNF}. In particular,
when $K$ is totally real, then this is, up to scaling, the unique continuous idèle class character (\cite[Corollary 2.4, Example 2.7]{BBBM-QCNF}). The ground fields in our examples in \cref{subsec:monic_ex} and \cref{subsec:Example-QC-OK-real-points} are indeed totally real.
For instance, when $K=\Q$, the cyclotomic character is given by $\tr_\p=\id$. Then 
$\chi_\p= \log_p$  is  defined by
$\log_p(p)=0$, since $\chi_q(p)=0$ for all $q\neq p$ by continuity.
Similarly, we also have $\chi_q(q) = -\log_p(q)$.
When $K$ is a real quadratic number field and $p$ splits in
$K$, then, identifying $K_\p$ with $\Q_p$ for all $\p\mid p$, we can again
take $\tr_\p=\id$, and we find $\chi_\p= \log_p$ with $\log_p(p)=0$; for
$\q\nmid p$, we have $\chi_\q(\pi_q) = -\log_p(N_{K_\q/\Q_q}(\pi_q))$.

\subsection{Local heights when $\chi_v$ is
unramified}\label{subsec:intro-heights-away-p}
We briefly review the construction of $h_v$ for a prime $v=\q$ of $K$ such that
$\chi_\q$ is unramified, for instance a prime $\q\nmid p$. 
Given two divisors $D_1,D_2\in \Div^0(X_\q)$ with disjoint support, we
denote by $\calD_1$ and $\calD_2$ extensions of these divisors to 
a regular model of $X_\q$ such that the 
intersection multiplicity $(\calD_i\cdot V)\in \Q$ vanishes
for all vertical divisors $V$.

\begin{proposition}{\cite[Proposition 1.2]{Coleman-Gross-Heights}}\label{P:htaway} 
For a prime $\q$ such that $\chi_\q$ is unramified, there is a unique local height pairing $h_\q$ in
  the sense of Definition \ref{def:local-height}. It is
defined by
\begin{equation}\label{eq:intmult}
  h_\q(D_1,D_2)=\chi_{\q}(\pi_\q)\cdot (\calD_1\cdot \calD_2)\,.
\end{equation}
\end{proposition} 

We describe the possible values $h_\q(\infty_--\infty_+, D_2)$ on even
degree hyperelliptic curves in~\S\ref{sec:heights_away}. These are crucial in our linear quadratic Chabauty method in Sections~\ref{sec:QC-over-Z} and~\ref{sec:QC-NF}.

\subsection{Local heights when $\chi_v$ is
ramified}\label{subsec:intro-heights-at-p}

When $\chi_v$ is ramified, then $v=\p\mid p$ and one can define a local height pairing $h_\p$ by the Coleman integral 
\begin{equation}\label{hp}
h_\p(D_1,D_2)=\tr_\p\left(\int_{D_2}\omega_{D_1}\right)\,,
\end{equation}
were $\omega_{D_1}$ is a differential of the third kind  whose residue
divisor is $D_1$ (see~\cite{Coleman-Gross-Heights} or~\cite[Section
2]{BBHeights}). There are many such differentials, and one uses the choice
of a subspace $W_\p \subset \hdr(X_\p/K_\p)$, complementary to the
holomorphic differentials, to fix one of them  as
in~\cite[Definition~2.7]{BBHeights}. We write $\omega_{D_1}$ for this differential.

\begin{proposition}{\cite[Proposition
  5.2]{Coleman-Gross-Heights}}\label{P:htabove}
  Let $\p\mid p$ be a prime such that $\chi_\p$ is ramified and let $W_\p \subset \hdr(X_\p/K_\p)$ be a complementary subspace that is
  isotropic with respect to the cup product.
  Then~\eqref{hp} defines 
  a local height pairing in the sense of Definition \ref{def:local-height}.
\end{proposition}

\section{Local heights for linear quadratic Chabauty on hyperelliptic
curves}\label{sec:heights_QC}
 In this section, we discuss the local heights that we use later on
 for our linear quadratic Chabauty method. 
We consider an even degree hyperelliptic curve $X\colon y^2
=f(x)$ over $K$ with $f\in \O_K[x]$ squarefree and monic.
Let $ \mathcal{U}=\Spec \O_K[x,y]/(y^2-f(x))$ and let
$\chi$ be a continuous idèle
class character~\eqref{idchar} such that $\chi_\p$ is ramified for at least
one prime $\p\mid p$ of $K$. We assume for simplicity that $p$ is totally split in $K$,
see the beginning of \S\ref{subsec:Thm-QC-OK-points}.

The key idea of any variant of the  quadratic Chabauty method is to use
bilinearity of the global height pairing and to analyze certain local
heights.  Different variants of the method use
different heights.
As discussed in the introduction, the local heights of interest in the present article are of the form 
$h_v(\infty_- - \infty_+, P-Q)$, with $P,Q\in X_v(K_v)$,
where $X_v= X\otimes K_v$.

\subsection{Local heights for linear Quadratic Chabauty: the ramified
case}\label{subsec:htsQCram}
Suppose that $\chi_\p$ is ramified and fix a complementary subspace $W_\p$. 
Extending and simplifying an earlier algorithm for odd degree hyperelliptic
curves due to Balakrishnan and Besser~\cite{BBHeights}, the authors show
in~\cite{p-adicheightsHC} how to compute $h_\p$ on $X_\p$, which we view as
a curve over $\Q_p$. 
We briefly describe how to compute 
$h_\p(\infty_- -\infty_+,R-S)$ for affine points $R, S\in X_\p(\Q_p)$.
Since the residue divisor of $\frac{x^gdx}{y}$ is precisely
$\infty_- -\infty_+$,~\cite[Corollary~3.14]{p-adicheightsHC} implies
\begin{equation}\label{eq:height-at-p-infinities}
  h_\p(\infty_-
  -\infty_+,R-S)=\tr_\p\left(\int_S^R\frac{x^gdx}{y}-\sum_{i=0}^{g-1} u_i
  \int_S^R\frac{x^idx}{y}\right)
\end{equation}
where $u_0,\ldots,u_{g-1}\in \Q_p$ are used to normalize 
$\omega_{\infty_- -\infty_+}$ with respect to $W_\p$; here $\tr_\p$ is as in~\eqref{ellfact}. Because of our assumption that $p$ is split in $K$, the map $\tr_\p$ is multiplication by a scalar. 
However,~\eqref{eq:height-at-p-infinities} remains true without this assumption.
The numbers $u_0,\ldots,u_{g-1}$ are defined in~\cite[Lemma~3.13]{p-adicheightsHC}; they depend only on  $X_\p$, $W_\p$ and the action of Frobenius on  
$\hdr(X_\p)$,
but
not on $R$ or $S$. We can compute  $u_0,\ldots,u_{g-1}$ 
efficiently using~\cite[\S3.2, Steps (vi), (vii)]{p-adicheightsHC}. The necessary Coleman
integrals can be computed efficiently using Balakrishnan's algorithm \cite{Jen-Even-Degree-CI}.

\subsection{Local heights for linear Quadratic Chabauty: the unramified
case}
\label{sec:heights_away}

Suppose that $\chi_\q$ is unramified.
One can compute 
$h_\q(\infty_- - \infty_+, P-Q)$ using algorithms from \cite{Hol12, Mul14, Raymond-David-Steffen}. 
We now find the possible values that 
$h_\q(\infty_- - \infty_+, P-Q)$ can take for $P,Q\in \calU(\O_\q)$. In
particular, there are only finitely many such values.
Recall from~\eqref{eq:intmult} that the local height on $X_\q$ can be
expressed in terms of intersection theory on a regular model. As
in~\cite{BBM1}, we use a specific regular model, constructed as follows:
Let $F(X,Z)$ be the degree~$2g+2$ homogenization of $f$. Then the equation
$Y^2=F(X,Z)$ defines a model of $X_\q$ in the weighted projective plane over
$\O_\q$, with weight~1 attached to $X$ and $Z$ and weight $g+1$ attached to $Y$.
Let $\mathcal{X}$ be  a strong desingularization  of this model (meaning
that this desingularization is an isomorphism outside the singular locus,
see \cite[Definition 8.3.39]{Liu}).
We denote the set of irreducible components of multiplicity~1 of its
special fiber $\mathcal{X}_\q$ by $V_1(\mathcal{X})$. Since all points in
$X(K_\q)$ reduce to a smooth point of $\mathcal{X}_\q$, there is a
natural specialization map
$$
\sp\colon X(K_\q)\to V_1(\mathcal{X}).
$$
Let $D(\mathcal{X})$ denote the free abelian group on $V_1(\mathcal{X})$
and let $M$ be the intersection matrix of $\mathcal{X}_\q$.
The Moore--Penrose pseudoinverse of $-M$ induces a bilinear pairing $L^+$
on $D(\mathcal{X})$.  

\begin{lemma}\label{L:away}\hfill
\begin{enumerate}[(a)]
    \item \label{L:away1}
Let $R,S\in \mathcal{U}(\O_\q)$. Then we have
$$
h_\q(\infty_--\infty_+, R-S) = L^+(\sp(\infty_-)-\sp(\infty_+),
\sp(R)-\sp(S))\chi_\q(\pi_\q)\,.
$$
\item\label{L:away2} The set
\[
T_\q = \{L^+(\sp(\infty_-)-\sp(\infty_+),
\sp(R)-\sp(S))\,\colon \,R,S\in \mathcal{U}(\O_\q)\}
\]
is finite.
\item\label{L:away3} If $\mathcal{X}_\q$ has only one component, then
  $T_\q=\{0\}$. In particular, this holds when $X_\q$ has good reduction. 
\item\label{L:away4} 
If $f$ is monic, $\q$ is not above~2 and if the reduction of $f$ is not a square, then
$T_\q= \{0\}$.
\end{enumerate}
\end{lemma}
\begin{proof}
Let $P,Q,R,S\in X(K_\q)$ such that $\{P,Q\}\cap \{R,S\} = \varnothing$.  
  By~\eqref{eq:intmult} and
  by~\cite[Sections~2,~3]{Gross-local-heights}, the local height $h_\q(P-Q,  
  R-S)$ is equal, up to a constant, to the real-valued N\'eron-symbol
  between $P-Q$ and $R-S$. The latter is analyzed, for instance,
  in~\cite[\S7]{dJH15}. In particular, it follows from~\cite[
  Corollary~7.5, Proposition~7.4 and its proof]{dJH15} that
$h_\q(P-Q, R-S)$ is equal to
  \[\left((P_{\mathcal{X}}-Q_{\mathcal{X}})\cdot
  (R_{\mathcal{X}}-S_{\mathcal{X}}) +L^+(\sp(P)-\sp(Q),
\sp(R)-\sp(S))\right)\chi_\q(\pi_q)\,,
\]
  where the subscript $\mathcal{X}$ denotes the Zariski closure in
  $\mathcal{X}$ and the first term is the $\Q$-valued intersection multiplicity on
  $\mathcal{X}$. If
  $P=\infty_-, Q=\infty_+$ and if $R,S\in \calU(\O_\q)$, then this first term vanishes, proving~\eqref{L:away1}.
To prove~\eqref{L:away2}, it suffices to note that $V_1(\mathcal{X})$ is finite.
When $\infty_-$ and $\infty_+$ specialize to the same component in
  $V_1(\mathcal{X})$, then
$ L^+(\sp(\infty_-)-\sp(\infty_+), \sp(R)-\sp(S))$  is trivial. This proves~\eqref{L:away3}. It
  also proves~\eqref{L:away4}, since if $f$ does not reduce to a square,
  then there is a unique component of $\mathcal{X}_\q$ that contains the
  images of points of good reduction modulo $\pi_q$, and if $f$ is monic, then both points at infinity have good reduction.
\end{proof}
\begin{remark}\label{alg:htaway}
In practice, it is easy to determine the set $T_\q$ using \texttt{Magma}'s command {\tt RegularModel}, but the computation of the regular model does not always succeed. If it does, then the command
{\tt
PointOnRegularModel} can be used to find the components $$\sp(\infty_-),\;
  \sp(\infty_+),\; \sp(R),\; \sp(S),$$ so that we only need to compute all
  values $L^+(\sp(\infty_-)- \sp(\infty_+),
\Gamma -\Gamma')$, where $\Gamma$ and $\Gamma'$ run through $V_1(\mathcal{X})$.
\end{remark}

\begin{remark}\label{Tqsimple}
The set $T_\q$ plays a crucial part in our description of the linear quadratic Chabauty method for integral points below. We stress that it is much simpler to describe and to compute than the corresponding sets~\cite[(5.2)]{BBM1} in the odd degree case and that for many examples it is trivial.
\end{remark}

\section{Linear quadratic Chabauty for $\Z$-integral points}\label{sec:QC-over-Z}

In this section, we introduce the linear quadratic Chabauty method to
compute the integral points on certain even degree hyperelliptic curves
over $\Q$. Let $$X\colon y^2=f(x)=d_{2g+2}x^{2g+2}+\cdots+d_0$$ be a hyperelliptic  curve over $\Q$ such that $f\in \Z[x]$ is a squarefree polynomial 
of degree $2g+2$ and $d_{2g+2}$ is a square in $\Z$. Using the obvious
change of variables defined over $\Q$, we may and will assume that $f$ is
monic; integral points are preserved. 
We assume that the rank $r$ of the Mordell-Weil group $J(\Q)$ of the Jacobian $J/\Q$ of $X$ equals $g$ and that $p$ is a prime of good reduction such that the closure of $J(\Q)$ in $J(\Q_p)$ has finite index,
so that we cannot use the classical method of Chabauty and Coleman. 
We set 
\[
\mathcal{U}=\Spec \Z[x,y]/(y^2-f(x));
\]
our goal is to compute $\mathcal{U}(\Z)$, the set of integral points in $X(\Q)$.

\subsection{Linear quadratic Chabauty for $\Z$-integral points: theory} \label{subsec:QC-Z-integral}
We now state and prove our result on linear quadratic Chabauty for $\Z$-integral points, and we show how it implies Theorem \ref{T:QC-integrals-pts-Z}.

\begin{theorem}\label{T:QCZ-precise}
Let $X/\Q\colon y^2=f(x)$ be a hyperelliptic curve such that $f\in \Z[x]$
  is a squarefree polynomial of degree $2g+2>3$ whose leading coefficient
  is a square in $\Z$. Let $\omega_i\colonequals \frac{x^idx}{y}$ and let
  $J$ be the Jacobian of $X$. Suppose that $r\colonequals \rk(J/\Q) = g$
  and that the closure of $J(\Q)$ has finite index in $J(\Q_p)$.  Let $p$
  be a prime of good reduction for $X$ and let $h=\sum_v h_v$ be the
  Coleman-Gross $p$-adic height pairing with respect to the cyclotomic
  idèle class character and a complementary subspace $W_p\subset
  \hdr(X/\Q_p)$, isotropic with respect to the cup product pairing, which we assume to be given explicitly. Assume that we know an integral point $Q\in \mathcal{U}(\Z)$.
 \begin{enumerate}[(a)]
     \item\label{precise1} There are constants 
$\alpha_0,\ldots,\alpha_{g-1}\in \Q_p$ such that for all $a\in J(\Q)$ we have
\[
h(\infty_- - \infty_+,a)=\alpha_0\int_{a}\omega_0 +\cdots+\alpha_{g-1}\int_{a}\omega_{g-1}\,.
\]
\item\label{precise2}
The function $\rho\colon \calU(\Z_p)\to \Q_p$ defined by
\[
\rho(P)\colonequals  \alpha_0\int_{Q}^P\omega_0 +\cdots+ \alpha_{g-1}\int_{Q}^P\omega_{g-1} - h_p(\infty_--\infty_+,P-Q)
\]
is locally analytic, computable and not identically~0.
\item\label{precise3}
Let $T=\{\sum_{q\ne p}-l_q\log_p(q)\,\colon \, l_q\in T_q\}$, where $T_q$ is defined in Lemma~\ref{L:away}~\eqref{L:away2}. Then $T$ is finite and computable, and we have 
$\rho(P)\in T$ for all $P\in \mathcal{U}(\Z)$.
 \end{enumerate}
\end{theorem}
\begin{proof}
   Since the closure of $J(\Q)$ has finite index in $J(\Q_p)$,
  the map $\log\colon J(\Q)\otimes\Q_p\to
  \mathrm{H}^0(X_{\Q_p},\Omega^1)^{\vee}$ is an isomorphism, so it follows that the linear maps
 $f_i\colon J(\Q)\rightarrow \Q_p$ given by $f_i(a)=\int_0^{a}\omega_i$ induce a basis of the space of linear maps $J(\Q)\otimes \Q_p\rightarrow
  \Q_p$. The map 
  $$\lambda(a)\colonequals h(\infty_- - \infty_+,a)$$
  belongs to this space, proving~\eqref{precise1}.

 Part~\eqref{precise2} follows from~\eqref{eq:height-at-p-infinities}.
For part~\eqref{precise3}, we recall that $T_q=\{0\}$ for almost all prime
  numbers $q\neq p$, by Lemma \ref{L:away} \eqref{L:away3}. Therefore, all
  sums in the definition of  the set $T$ are finite. Further, all $T_q$ are finite by Lemma \ref{L:away} \eqref{L:away2}, and computable, using, for example, algorithms from \cite{Raymond-David-Steffen}, so $T$ is finite and computable. For $P\in \mathcal{U}(\Z)$, we have 
\[
\rho(P)=h(\infty_- - \infty_+,P-Q)-h_p(\infty_- - \infty_+,P-Q)=\sum_{q\neq p}h_q(\infty_- - \infty_+,P-Q).
\]
Hence the final assertion of part~\eqref{precise3} follows from 
Lemma \ref{L:away} \eqref{L:away1}.
\end{proof}
\begin{remark}\label{r:no-int-point}

We can replace $\rho$ by the function
\[
\rho(P)\colonequals  \alpha_0\int_{\iota(P)}^P\omega_0 +\cdots+ \alpha_{g-1}\int_{\iota(P)}^P\omega_{g-1} - h_p(\infty_--\infty_+,P-\iota(P))
\]
if no integral point is known; Theorem~\ref{T:QCZ-precise} remains valid in this case.
\end{remark}
\begin{remark}\label{R:}
  If $p$ is a prime of good ordinary reduction, then there is a canonical
  choice for the subspace $W_p$, namely the unit root subspace with respect
  to Frobenius. However, our method works for any complementary 
  $W_p$ {that is isotropic with respect to the cup product pairing}; see also the discussion in~\cite[\S3.1, Step (iv)]{p-adicheightsHC}. 
\end{remark}

\begin{proof}[Proof of Theorem~\ref{T:QC-integrals-pts-Z}]
Suppose that $X/\Q$ satisfies the conditions of Theorem~\ref{T:QC-integrals-pts-Z}.
Theorem~\ref{T:QC-integrals-pts-Z} follows at once from
  Theorem~\ref{T:QCZ-precise} and Remark~\ref{r:no-int-point} if the
  closure of $J(\Q)$ has finite index in $J(\Q_p)$.
By construction, the function $\rho$ is not identically~0. 
  If this condition does not hold, we can use Chabauty--Coleman to compute
  $X(\Q)$ and thus $\mathcal{U}(\Z)$.
\end{proof}

\subsection{Computing $\Z$-integral points using linear quadratic Chabauty} \label{subsec:QC-alg-Z-integral} We keep the notation of~\S\ref{subsec:QC-Z-integral} and we assume
that the conditions of Theorem~\ref{T:QCZ-precise} are satisfied. Then we
get the 
following algorithm for the
computation of $\mathcal{U}(\Z)$. 
We denote by $\mathrm{red}\colon X(\Q_p)\longrightarrow X(\F_p)$ the map that  sends $P$ to $\mathrm{red}(P)=\bar{P}$, its reduction modulo $p$. A well-behaved uniformizer at a point $P\in X(\Q_p)$ is a
uniformizer $t_P$ at $P$ that reduces to a uniformizer at 
$\bar{P}\in X(\F_p)$. Choosing a
well-behaved uniformizer at any point of the residue disc 
\begin{equation*}\label{resdisc}
  D(P) = D(\bar{P})  \colonequals \mathrm{red}^{-1}(\bar{P})
\end{equation*}
induces a bijection of $D(P)$ with $p\Z_p$ (see for instance~\cite{Lorenzini-Tucker}).

\begin{alg}\label{alg:QC-Z}{\bf (Linear quadratic Chabauty for integral points)}

  \begin{itemize}
    \item \textbf{Input}: 
      \begin{itemize}
        \item       A hyperelliptic curve $X/\Q\colon y^2=f(x)$,
      where $f\in \Z[x]$ is a monic squarefree polynomial of degree
      $2g+2>3$ such that the
      Mordell--Weil rank $r$ of the Jacobian $J$ of $X$ over $\Q$ is equal
      to $g$; 
    \item  a prime $p$ of good reduction for $X$ and such that the $p$-adic closure of $J(\Q)$ has finite index
      in $J(\Q_p)$;
    \item {a complementary subspace $W_p\subset
  \hdr(X/\Q_p)$, isotropic with respect to the cup product pairing}. 
      \end{itemize}
    \item \textbf{Output}: $\mathcal{U}(\Z)$, where 
$\mathcal{U}=\Spec \Z[x,y]/(y^2-f(x))$,  or FAIL.
  \end{itemize}
\begin{enumerate}[(i)]
\item\label{a:qc-small} Find all integral points $P\in \mathcal{U}(\Z)$ up to some height bound.
\item\label{a:qc-q-heights} For each bad prime $q$, compute the set $T_q$ using Remark~\ref{alg:htaway}.
\item\label{a:qc-hts}
Find $g$ independent points $a_1,\ldots, a_g\in J(\Q)$ and compute
    $h(\infty_- - \infty_+,a_i)$\\ for $1\le i \le g$ using the
    algorithms from 
    {\cite[\S\S3.1, 3.2]{p-adicheightsHC}} for $h_p$, and using one of the algorithms from~\cite{Hol12, Mul14, Raymond-David-Steffen}
for $h_q$ for primes $q\neq p$. 
\item\label{a:qc-express-linear-maps} 
Using Step~\eqref{a:qc-hts}, compute
$\alpha_0,\ldots,\alpha_{g-1}\in \Q_p$ such that for all $a\in J(\Q)$ we have
\[
h(\infty_- - \infty_+,a)=\alpha_0\int_{a}\omega_0 +\cdots+\alpha_{g-1}\int_{a}\omega_{g-1}.
\]
Define the function $\rho\colon \calU(\Z_p)\to \Q_p$
    by
\[
\rho(P)\colonequals  \alpha_0\int_{Q}^P\omega_0 +\cdots+ \alpha_{g-1}\int_{Q}^P\omega_{g-1} - h_p(\infty_--\infty_+,P-Q).
\]
\item\label{a:qc-rho-in-z} In each affine residue disc, fix a point $P_0\in
  X(\Q_p)$. Use a well-behaved uniformizer $t_{P_0}=t_{P_0}(z)$ at $P_0$ to
    parametrize the points $P$ inside this residue disc and express  $\rho(P)$ as a power series in $z$. 
\item\label{a:qc-solve-in-z} For each $\tau\in T$, find the set $A_{\tau}\subset
  \calU(\Z_p)$ of all points corresponding to a zero of $\rho(z)=\tau$. 
\item\label{a:qc-mw-sieve} If $A\colonequals \bigcup_{\tau\in T}A_{\tau}$ does not equal the set computed in Step~\eqref{a:qc-small}, use the Mordell-Weil sieve 
to separate which points in $A$ are rational and which are not as
    in~\cite{BBM2}. If this does not succeed return FAIL. Otherwise 
    return the integral points among the remaining ones.
\end{enumerate}
\end{alg}

\begin{remark}\label{r:QC-over-Z-algorithm-remark}\hfill
\begin{enumerate}[(a)]
\item In Step~\eqref{a:qc-hts}, we preferably choose points in $J(\Q)$
  represented by divisors whose support consists only of points in
    $\calU(\Z)$. However, it is still possible to compute $\alpha_0,\ldots,\alpha_{g-1}$ if we need to use other representatives.
\item 
We compute the power series in Step~\eqref{a:qc-rho-in-z} as follows. We write 
\begin{equation}\label{eq:express-as-power-series}
\rho(P)=  \alpha_0\int_{Q}^{P_0}\omega_0 +\cdots+ \alpha_{g-1}\int_{Q}^{P_0}\omega_{g-1} - h_p(\infty_--\infty_+,P_0-Q)
\end{equation}
\[\qquad\qquad
+\alpha_0\int_{P_0}^P\omega_0 +\cdots+ \alpha_{g-1}\int_{P_0}^P\omega_{g-1} - h_p(\infty_--\infty_+,P-P_0).
\]
All the quantities in the first line of~\eqref{eq:express-as-power-series} are independent of $P$, and we have
    already explained how to compute them. Recalling that we have 
\[
h_p(\infty_--\infty_+,P-P_0)= \int_{P_0}^P\omega_g-\left( \beta_0\int_{P_0}^P\omega_0 +\cdots+ \beta_{g-1}\int_{P_0}^P\omega_{g-1}\right)
\]
for some coefficients $\beta_0,\ldots,\beta_{g-1}\in \Q_p$, all the integrals that appear in the second line of~\eqref{eq:express-as-power-series} are tiny integrals, (see \cite[\S3.2 and Algorithm 3]{Jen-Even-Degree-CI}). As tiny integrals are computed integrating term by term,  if a point $P$ is described using a parameter $z$, the function $\rho$ is indeed expressed as a power series in $z$.
If we use $\rho$ as suggested in Remark \ref{r:no-int-point}, then we compute the integral $\int_{\iota(P)}^P \omega_i$ for any $0\leq i \leq g$ as
\[\int_{\iota(P)}^P \omega_i = \int_{\iota(P)}^{\iota(P_0)}\omega_i+\int_{\iota(P_0)}^{P_0}\omega_i+\int_{P_0}^P \omega_i= \int_{\iota(P_0)}^{P_0}\omega_i+2\int_{P_0}^P \omega_i\,.\]

\item We can get a smaller set $A$ of zeros by noting that 
the sets $T_q$ are actually larger than necessary. In fact, it suffices to compute the set of all values
\[L^+(\sp(\infty_-)- \sp(\infty_+),
\Gamma -\sp(Q))\,,
\]
where $\Gamma$ runs through $V_1(\mathcal{X})$.
\item\label{r:no-integral-points-algorithm} Recall from Remark~\ref{r:no-int-point} that we do not need
to have an integral point $Q$ 
    available, since we can use the function
\[
\rho(P)\colonequals  \alpha_0\int_{\iota(P)}^P\omega_0 +\cdots+
    \alpha_{g-1}\int_{\iota(P)}^P\omega_{g-1} -
    h_p(\infty_--\infty_+,P-\iota(P))\,,
\]
where $\iota$ is the hyperelliptic involution.
In this case, we could get smaller sets $T_q$ by computing how $\iota$ acts
    on the special fiber $\mathcal{X}_q$, but we have not implemented this.
\end{enumerate}
\end{remark}

We have implemented Algorithm~\ref{alg:QC-Z}, and our code is available
from~{\url{https://github.com/steffenmueller/LinQC}}. Most of our implementation is
in {\tt SageMath}~\cite{Sage},
but Step~\eqref{a:qc-q-heights} and the
computation of local heights at $q\ne p$ in Step~\eqref{a:qc-hts} are implemented in {\tt
Magma}~\cite{Magma}, since they require the computation of a regular model. 

\subsection{Precision estimates} \label{subsec:prec-qc-for-Z}
We now analyze the loss of precision in Algorithm~\ref{alg:QC-Z}. Steps
\eqref{a:qc-small} and \eqref{a:qc-q-heights} are exact. When computing heights away from
$p$, we get $\Q$-valued intersection numbers, multiplied by $\log_p(q)$ (or, more
generally, $\chi_p(q)$), which we can compute to any desired precision. Thus,  the
precision analysis for Step \eqref{a:qc-hts} depends only on the precision analysis for the local height at $p$,
which {is discussed in \cite[\S3.4]{p-adicheightsHC}}; there we also
analyze precision for the Coleman integrals in Step
\eqref{a:qc-express-linear-maps}. The precision in Step \eqref{a:qc-express-linear-maps}
also depends on $k\colonequals \ord_p(\det N)$, where $N$ is the matrix
$N=\left(\int_{a_j}\omega_i\right)_{i,j}$ 
because we use
the inverse of $N$ to compute the constants $\alpha_0,\ldots,\alpha_{g-1}$. Thus, if we
want to compute $\alpha_0,\ldots,\alpha_{g-1}$ up to precision $p^m$, we have to compute
all previous quantities up to precision $p^{m+k}$. Note that $k$ depends only on the set
$\{a_1,\ldots,a_g\}$, which we can choose. Therefore, we can precompute the value $k$ and
deduce the required precision for the other parts of the algorithm. 
In Step \eqref{a:qc-rho-in-z} we consider tiny integrals of differentials which are
holomorphic in these residue discs, where one endpoint is given by a parameter. Hence the
possible loss of precision can be bounded as in~\cite[Proposition 18]{BBK}.
Since Step
\eqref{a:qc-mw-sieve} uses exact information at primes $v\ne p$, it does not require a precision analysis.

In Step \eqref{a:qc-solve-in-z}, we locate the zeros in the unit disc of a $p$-adic power series
$F$, truncated at
finite level $M$ and with coefficients correct to $m$ digits of $p$-adic precision.
The extent to which the resulting zeros of this polynomial coincide with the zeros of the
power series itself depends on $m$, $M$, and on the
valuations of the coefficients of $F$. If we have a lower bound for these valuations, then
we can use \cite[Proposition~3.11]{CCExperimentsGenus3} to determine the minimum values of $m$
and $M$ so that the zeros are provably correct to the desired
precision. 
In general, it can be difficult to obtain bounds on these valuations (see for
instance~\cite[Section~4]{BDMTV2}).
Fortunately, in our case $F$ is of a rather simple nature: up to constants, it is a linear combination of
tiny integrals of differentials $\omega_i$ for $i\in \{0,\ldots,g\}$ in
residue discs
where these do not have poles. Hence the only negative valuations come from integrating
terms of the form $a_jz^j$, where $p\mid(j+1)$, as in~\cite[Lemma~3.10]{CCExperimentsGenus3}.

\subsection{Example: A curve of genus 2}\label{subsec:monic_ex}

We illustrate the method from \cref{subsec:QC-alg-Z-integral} with the following example. Consider the curve $X/\Q$ defined by
\[
X\colon y^2 = f(x) = x^6 + 2x^5 - 7x^4 - 18x^3 + 2x^2 + 20x + 9
\]
with {\tt LMFDB}-label {\tt 6081.b.164187.1}. By
\url{https://www.lmfdb.org/Genus2Curve/Q/6081/b/164187/1}, the rank of its
Jacobian $J$ over $\Q$ is 2. Hence we may apply linear quadratic Chabauty, and we will use this to prove that the set of integral points on $X$ is
\begin{equation}\label{known-pts}
\{(0,\pm3),(1,\pm3),(-1,\pm1),(-2,\pm3),(-4,\pm37)\}.
\end{equation}
Since the geometric endomorphism ring is $\Z$, the quadratic Chabauty method for rational points is not applicable to this curve, and its rational points have not been determined.
Fix $Q=(-1,1)\in \calU(\Z)$. Let $p$ be a prime of good ordinary reduction
for $X$. Using Lemma \ref{L:away}, we compute that for any prime $q\neq p$
and any $P\in\calU(\Z)$, we have 
\[
h_q(\infty_- - \infty_+, P-Q)=0.
\] 
Hence we have $T=\{0\}$.
Namely, for $q\ne 2$, $f$ is not a square modulo $q$. At~2, the given model is not minimal at 2;  the minimal model
$$X^{\min}\colon y^2 + (x^3 + x^2 + 1)y = -2x^4 - 5x^3 + 5x +     2$$
is smooth at~2 and any isomorphism $X\to X^{\min}$ maps 2-integral points to 2-integral points. Therefore the local height at~2 is trivial as well. 

The next step is to express the Coleman-Gross height in terms of Coleman
integrals, i.e., to compute $\alpha_0,\alpha_1$ from Algorithm
\ref{alg:QC-Z}, Step \eqref{a:qc-express-linear-maps}. We first need to
compute $h(\infty_- - \infty_+,a_1)$ and $h(\infty_- - \infty_+,a_2)$ for two
linearly independent $a_1,a_2\in J(\Q)$. As explained 
in Remark \ref{r:QC-over-Z-algorithm-remark}~(a), we prefer to use two $\Z$-linearly independent elements of
$J(\Q)$ that have a representative of the shape $P-R$ for $P,R\in \calU(\Z)$ because then it is easier to compute the heights\- $h(\infty_- - \infty_+,P-R)$. We can indeed find such representatives for this curve. Namely, if we denote $Q_0=(0,3)$ and $Q_1=(1,3)$, then $[Q_0-Q]$ and $[Q_1-Q]$ are $\Z$-linearly independent in $J(\Q)$.

By the above, we have $h(\infty_- - \infty_+,P-R)=h_p(\infty_- -
\infty_+,P-R)$ for $P,R\in \calU(\Z)$. Taking $P=Q_0, R=Q$, and $P=Q_1,
R=Q$, respectively, we get two linear equations: 
\[
h(\infty_- -\infty_+,Q_i-Q)=h_p(\infty_-
-\infty_+,Q_i-Q)=\alpha_0\int_Q^{Q_i}\omega_0+\alpha_1\int_Q^{Q_i}\omega_1\,,\quad
i=0,1\,.
\]

We use {the good, ordinary prime $p=7$, and the unit root subspace
$W_7$}. Solving the system above, we compute 
\[
\alpha_0=5 + 4\cdot7 + 6\cdot7^2 + 7^3 + 6\cdot7^4 + O(7^5),\;
\alpha_1=6 + 3\cdot7 + 5\cdot7^2 + 3\cdot7^3 + 4\cdot7^4 + O(7^5)\,.
\]

Now we define the function $\rho\colon \calU(\Z_p)\to \Q_p$ by \[
\rho(P)=\alpha_0\int_{Q}^P\omega_0 + \alpha_1\int_{Q}^P\omega_1 -
h_p(\infty_--\infty_+,P-Q)\,.
\] 

As explained in Step~\eqref{a:qc-rho-in-z} of Algorithm \ref{alg:QC-Z}, in each residue disc,
we analyze the zeros of $\rho$ by fixing a point $P_0$ and then choosing   a
well-behaved uniformizer at $P_0$ to establish a bijection between the
residue disc $D(P_0)$  and $7\Z_7$.
All affine residue discs modulo 7 contain exactly one known integral point
from~\eqref{known-pts}. Hence, for each residue disc, we choose that
integral point as a representative. To prove that this integral point is
the only one in its residue disc, it suffices to show that $z=0$ is the only zero
of the power series expansion of $\rho$ in the disc. 
This indeed happens for all known integral points on $X$. 

For example, let us consider the point $Q_0=(0,3)$. All points in $D(Q_0)$ have $x$-coordinate equal to $7z$ for some $z\in \Z_7$. We compute that 
\[
\rho(z)=(6\cdot 7 + O(7^2))z + (6\cdot 7^3 +  O(7^4))z^2 + (6\cdot 7^3 +
O(7^4))z^3 + (7^4+O(7^5))z^4 + O((7z)^5)\,.\]
By Strassmann's theorem (see, for example, \cite[Theorem 4.2.4]{Gouvea-p-adic-numbers}), this power series has at most one zero. Since we already know that $z=0$ is a zero, we conclude that $Q_0$ is the only integral point on $X$ inside its residue disc. The same argument applies for all integral points and shows that the set of known integral points \eqref{known-pts} is the complete set of integral points on $X$.
The code for this example can be found
at~{\url{https://github.com/steffenmueller/LinQC}}.

\section[Linear quadratic Chabauty over number fields]{Linear quadratic Chabauty for integral points over number fields}\label{sec:QC-NF}

Let $K$ be a number field of degree $d$ with signature $(r_1,r_2)$; we
denote the rank of its unit group by $r_K=r_1+r_2-1$. In this section we
adapt the strategy from Section~\ref{sec:QC-over-Z} to compute the
$\O_K$-integral points on certain even degree hyperelliptic curves over $K$. Let
$X/K\colon y^2=f(x)$ be a hyperelliptic curve of genus $g$ such that $f\in \O_K[x]$
has no repeated roots. 
Let $J$ be the Jacobian of
$X$.
As in Section~\ref{sec:QC-over-Z}, denote 
\[
\mathcal{U}=\Spec\O_K[x,y]/(y^2-f(x));
\]
then $\mathcal{U}(\O_K)$ is the set of $\O_K$-integral points on $X$. 
The restriction of  scalars $\Res_{K/\Q}(X)$ is a $d$-dimensional variety,
so we  need at least $d$ linear conditions to cut out $\calU(\O_K)$.

Siksek~\cite{Siksek-Chabauty-NF} has extended Chabauty--Coleman to curves over number fields by using
all primes of $K$ above a fixed prime $p$ and 
the $gd$-dimensional variety $\Res_{K/\Q}(J)$, which has $gd$ linear
functionals to $\Q_p$,
given in terms of Coleman integrals.
Hence, if $$\rk(J/K)\le gd-d\,,$$ we can construct $d$ linear equations
satisfied by $X(K)$ in terms of Coleman integrals. While these equations
often cut out a finite set in practice, this is not always the case. For
instance, if $X$ is defined over $\Q$ and does not satisfy the Chabauty
rank condition over $\Q$, this method cannot succeed over any number field $K$, even if $X$ satisfies the rank condition over $K$. In addition, the method sometimes does not work for more subtle reasons, as explained in \cite{Dogra-Unlikely-Intersections-Chabauty-Kim}.
For more work on Chabauty methods based on restriction of scalars,
see~\cite{Tri}.

In~\cite{BBBM-QCNF}, Balakrishnan et al. extend the quadratic Chabauty
method from~\cite{BBM1} and combine it with Siksek's approach when $f\in
\O_K[x]$ is monic and of odd degree. 
One can show that there are at least $r_2+1$ linearly independent $p$-adic
height pairings (see~\cite[Section~2.2]{BBBM-QCNF}), corresponding to independent continuous $p$-adic idèle class characters of $K$, and each of these leads to an additional equation satisfied by the integral points. Hence the quadratic Chabauty method has a chance to work when
\begin{equation}\label{eq:rk_K}
    \rk(J/K)\leq dg-d+r_2+1=dg-r_K.
\end{equation}
As in Siksek's work,
a necessary condition for the method to cut out a finite set is that~\eqref{eq:rk_K} is satisfied for all subfields of $K$ over which $X$ is defined.

\subsection{Linear quadratic Chabauty for integral points over number fields:
theory}\label{subsec:Thm-QC-OK-points}
In the remainder of this section, we extend the linear quadratic Chabauty method introduced in
the previous section to the case where 
$f\in \O_K[x]$ is monic and of even degree $2g+2$. 
We assume \eqref{eq:rk_K} throughout.
For practical reasons, we consider a prime $p\in \Z$ of good reduction for
$X$ that splits completely in $K$, say $p\O_K=\p_1\cdots\p_d$. Otherwise, we
would have to compute Coleman integrals over proper extensions of $\Q_p$,
which is currently not implemented (this is work in progress due to
Best, Kaya and Keller).
However, as in~\cite{BBBM-QCNF}, our results
remain true for primes that do not split completely.

Denote by $\sigma_i$ the embedding $\sigma_i\colon
K\xhookrightarrow{}K_{\p_i}\simeq \Q_p$, for $1\leq i\leq d$. We obtain an
embedding $\sigma\colon K\xhookrightarrow{} K\otimes\Q_p\simeq \Q_p^d$
defined by $\sigma(u)=(\sigma_1(u),\ldots,\sigma_d(u))$ for $u\in K$. Via
this embedding, we identify $X(K\otimes\Q_p)\simeq X(\Q_p)^d$; there are
induced embeddings $\sigma\colon X(K)\xhookrightarrow{} X(K\otimes\Q_p)$
and $\sigma\colon \calU(\O_K)\xhookrightarrow{} \calU(\O_K\otimes\Z_p)$.
We also obtain embeddings $\sigma_i\colon J(K)\xhookrightarrow{} J(\Q_p)$,
which lead to $dg$ linear maps $J(K)\to\Q_p$ given by \[
\ell_{j,i}(a)=\int_{\sigma_i(a)}\omega_j,
\]
for $i\in\{1,\ldots,d\}$ and $j\in\{0,\ldots,g-1\}$. We also denote the
base-changed linear maps $\ell_{j,i}\colon J(K)\otimes \Q_p\to
\Q_p$. Because of our assumption \eqref{eq:rk_K} there are at least $r_K$ linear relations between the $\ell_{j,i}$'s in the space $(J(K)\otimes \Q_p)^{\vee}$.

As in \cite[Condition 4.1]{BBBM-QCNF}, from now on, we assume the following condition.

\begin{cond}\label{cond4}
We have $\langle \ell_{j,i} \rangle_{1\leq i\leq d, 0\leq j \leq g-1}=(J(K)\otimes \Q_p)^{\vee}$.
\end{cond}

We also recall that for each continuous idèle class character $\chi\colon
\mathbb{A}_K^{\ast}/K^{\ast}\to \Q_p$, there is a corresponding $p$-adic
height $h^{\chi}$, {depending on a choice of complementary 
subspace $W_\p \subset \hdr(X_\p/K_\p)$ (isotropic with respect to the cup
product) for each $\p\mid p$ such that
$\chi_\p$ is ramified. }
From now on, we assume that $\chi$ is ramified
over at least one prime $\p$ of $K$ above $p$ (see
\cref{subsec:intro-heights-at-p}) {and, for all $\p\mid p$, we take the same $W_\p$ for each
$\chi$ such that $\chi_\p$ is ramified}. Since we assume
Condition \ref{cond4}, we can express the linear map \[\lambda^{\chi}\colon
J(K)\to \Q_p, \hspace*{2mm}\lambda^{\chi}(a)=
h^{\chi}(\infty_--\infty_+,a),\] extended linearly to $J(K)\otimes \Q_p\to \Q_p$, as 
\begin{equation}\label{eq:heights-from-idele}
\lambda^{\chi}=\sum_{j=0}^{g-1}\sum_{i=1}^d \alpha_{j,i}^{\chi}\ell_{j,i}, \text{ for some $\alpha_{j,i}^{\chi}\in \Q_p$.}    
\end{equation}
We decompose $h^{\chi}\colon J(K)\to \Q_p$ into two parts 
\[h^{\chi}(a)=\sum_{\mathfrak{p}\mid p}h^{\chi}_{\mathfrak{p}}(D)+\sum_{\mathfrak{q}\nmid p}h^{\chi}_{\mathfrak{q}}(D)\,,\] 
where $D\in \Div^0(X)$ represents $a \in J(K)$.

For $D\in \Div^0(X)$ such that $\infty_{\pm}\notin \supp(D)$, we can rewrite~\eqref{eq:heights-from-idele} as
\begin{equation*}\label{eq:quadratic-chabauty-O_K-equations-from-heights}
\sum_{\mathfrak{q}\nmid p}h^{\chi}_{\mathfrak{q}}(\infty_- - \infty_+,D)=\sum_{j=0}^{g-1}\sum_{i=1}^d \alpha_{j,i}^{\chi}\int_{\sigma_i(D)}\omega_j-\sum_{i=1}^d h^{\chi}_{\mathfrak{p}_i}(\infty_--\infty_+,\sigma_i(D)).    
\end{equation*}

Suppose that there is an $\O_K$-integral point $Q\in \mathcal{U}(\O_K)$.
Then, as in Section~\ref{sec:QC-over-Z}, for $\mathfrak{p}\mid p$, the
function $P\mapsto h^{\chi}_{\mathfrak{p}}(\infty_--\infty_+,P-Q)$ is
locally analytic on $\calU(\O_{K_{\mathfrak{p}}})$, since it is given by
Coleman integrals. Moreover, by Lemma~\ref{L:away} for $\mathfrak{q}\nmid
p$, the function $$P\mapsto h^{\chi}_{\mathfrak{q}}(\infty_- -
\infty_+,P-Q)$$ takes values in a computable finite set $T_{\mathfrak{q}}$ when $P\in \mathcal{U}(\O_{K_{\q}})$, and $T_{\mathfrak{q}}=\{0\}$ for primes $\mathfrak{q}$ of good reduction. 
Hence we get $r_2+1$ relations 
\begin{equation*}\label{eq:quadratic-chabauty-O_K-equations-from-heights_p}
\sum_{j=0}^{g-1}\sum_{i=1}^d \alpha_{j,i}^{\chi}\int^{\sigma_i(P)}_{\sigma_i(Q)}\omega_j-\sum_{i=1}^d h^{\chi}_{\mathfrak{p}_i}(\infty_--\infty_+,\sigma_i(P)-\sigma_i(Q)) \in T_{\chi}
\end{equation*}
for $P\in \mathcal{U}(\O_K)$,
where $T_{\chi}$ is a computable finite set.
Together with $r_K$ linear relations of the shape 
\begin{equation}\label{eq:quadratic-chabauty-O_K-equations-from-Coleman-integrals}
  \sum_{j=0}^{g-1}\sum_{i=1}^d \beta_{j,i}^{(k)}\int_{\sigma_i(a)}\omega_j=0,    
\end{equation}
for $1\leq k\leq r_K$, we form a number of systems of $d$ linear equations
which we try to use to determine $\mathcal{U}(\O_K)$.

Using the above, we obtain the following generalization of Theorem \ref{T:QCZ-precise}.

\begin{theorem}\label{T:QC-NF-integrals-pts-full-statement} 
Let $X/K\colon y^2=f(x)$ be a hyperelliptic curve such that $f\in \O_K[x]$
  is monic, squarefree, and has degree $2g+2$. Suppose $\rk(J/K)\leq
  dg-r_K$. Let $p\in \Z$ be a prime that splits in $K$ such that all primes
  above $p$ in $K$ are primes of good  reduction for $X$. Assume
  Condition~\ref{cond4} and that there is point $Q\in
  \mathcal{U}(\O_K)$. 
  Fix independent continuous idèle
    class characters $\chi_{r_K+1},\ldots, \chi_{d}$ {and, for each
    prime $\p\mid p$, a complementary subspace $W_\p \subset \hdr(X_\p/K_\p)$  that is
  isotropic with respect to the cup product.}

Then there are $d$ computable locally analytic functions
  $\rho_1,\ldots,\rho_d\colon \calU(\O_K\otimes \Z_p)\rightarrow \Q_p$ with the following properties.
\begin{enumerate}[(i)]
  \item For $1\leq k\leq r_K$ and $\mathcal{P}=(P_1,\ldots,P_d)\in
    \calU(\O_K\otimes \Z_p)$, we have
\[
\rho_k(\mathcal{P})=\sum_{j=0}^{g-1}\sum_{i=1}^d
    \beta_{j,i}^{(k)}\int_{\sigma_i(Q)}^{P_i}\omega_j
\]
and $\rho_k(\mathcal{P})=0$ for $\mathcal{P}\in \sigma(\calU(K))$.
\item For $r_K+1\leq k\leq d$ and $\mathcal{P}=(P_1,\ldots,P_d)\in
    \calU(\O_K\otimes \Z_p)$, 
  we have 
\[
\rho_k(\mathcal{P})=\sum_{j=0}^{g-1}\sum_{i=1}^d
    \alpha_{j,i}^{\chi_{k}}\int_{\sigma_i(Q)}^{P_i}\omega_j-\sum_{i=1}^d
    h^{\chi_{k}}_{\mathfrak{p}_i}(\infty_--\infty_+,P_i-\sigma_i(Q))\,;
\]
     moreover, 
there is a finite and computable set $T_k$ such that
    $\rho_k(\mathcal{P})\in T_k$ for all $\mathcal{P}\in
    \sigma(\mathcal{U}(\O_K))$.
\end{enumerate}
\end{theorem}

\begin{proof}
The proof follows from the discussion above and by a straightforward modification of the proof of Theorem \ref{T:QCZ-precise}.
\end{proof}

By the same reasoning as in Remark \ref{r:no-int-point} and Remark \ref{r:QC-over-Z-algorithm-remark}\eqref{r:no-integral-points-algorithm}, it is not necessary to assume that we have found an $\O_K$-integral point on $X$.

\subsection{Computing $\O_K$-integral points}\label{subsec:alg-OK-points}

We now discuss how Theorem \ref{T:QC-NF-integrals-pts-full-statement} may be used to compute the $\O_K$-integral points in practice. We imitate the steps from Algorithm \ref{alg:QC-Z}, except for a few minor changes that we now explain. 

Steps \eqref{a:qc-small}, \eqref{a:qc-q-heights} and \eqref{a:qc-hts}, or
\eqref{a:qc-mw-sieve} are exactly as in Algorithm~\ref{alg:QC-Z}. Step
\eqref{a:qc-express-linear-maps} now has two parts. First, we find $dg-r_K$
linearly independent elements $a_1,\ldots, a_{dg-r_K}\in J(K)$. Using basic
linear algebra, we find $r_K$ linear relations amongst
$\int_{\sigma_i(a_s)}\omega_j$ for $1\leq i\leq d$, $0\leq j \leq g-1$ and
$1\leq s \leq dg-r_K$. In this way, we get relations of the form
\eqref{eq:quadratic-chabauty-O_K-equations-from-Coleman-integrals}.
Next, for each chosen idèle class character $\chi$, we express
$h^{\chi}(\infty_- -\infty_+, a)$ in terms of $\int_{\sigma_i(a)}\omega_j$,
again by computing all quantities involved for
$a\in\{a_1,\ldots,a_{dg-r_K}\}$. It is again a basic linear algebra task to
compute coefficients $\alpha_{j,i}^{\chi}$ as in
\eqref{eq:heights-from-idele}. 
From now on, we assume that the functions $\rho_1,\ldots,\rho_d$ have been computed.

Step \eqref{a:qc-rho-in-z} differs from the $\Z$-integral case because we
now consider $d$ residue discs at once, coming from the $d$ primes
$\p_1,\ldots,\p_d$. For each $\overline{\mathcal{P}}=(\overline{P_1},\ldots,\overline{P_d})\in \overline{X}(\F_p)^d$, we introduce the residue polydisc notation 
\[
  \mathcal{D}(\overline{\mathcal{P}})\colonequals \{(P_1,\ldots,P_d)\in X(K\otimes \Q_p)\colon P_k\equiv \overline{P_k}\pmod{\p_k}, \text{ for each $1\leq k\leq d$}\}.
\] 
For each  $\mathcal{P}\in \mathcal{D}(\overline{\mathcal{P}})$, we also use the
notation
$\mathcal{D}(\mathcal{P})\colonequals\mathcal{D}(\overline{\mathcal{P}})$
for the residue polydisc that contains $\mathcal{P}$. We use any point
$\mathcal{P}= (P_1, \ldots , P_d) \in \mathcal{D}(\overline{\mathcal{P}})$ and well-behaved uniformizers
$t_{\p_1}=t_{\p_1}(z_1),\ldots,$ $t_{\p_d}=t_{\p_d}(z_d)$ 
with respect to $\p_1,\ldots,\p_d$, respectively, to parametrize all points
in $\mathcal{D}(\mathcal{P})$ by tuples $(z_1,\ldots,z_d)\in (p\Z_p)^d$.

Next, we express the functions $\rho_1,\ldots,\rho_d$ in the parameters
$(z_1,\ldots,z_d)$ by splitting Coleman integrals and local heights above
$p$ using the point $\mathcal{P}$. Namely, we imitate Step
\eqref{a:qc-rho-in-z} of Algorithm \ref{alg:QC-Z}, as explained  in Remark~\ref{r:QC-over-Z-algorithm-remark} (b), considering residue discs with respect to all $\mathfrak{p}_i$, $i\in\{1,\ldots,d\}$. 
 
In Step \eqref{a:qc-solve-in-z}, we need to find solutions to systems of multivariable power series with separated variables. Namely, we get $d$ equations of the shape
$F_k(z_1,\ldots,z_d)=0$, where 
\begin{itemize}
\item $F_k\in \Q_p[[z_1,\ldots,z_d]]$;
\item $F_k(z_1,\ldots,z_d)=s_k+F_{k,1}(z_1)+\cdots+F_{k,d}(z_d)$, with $s_k\in \Q_p$, $F_{k,j}\in \Q_p[[z_j]]$, and $F_{k,j}(0)=0$, for all $1\leq j\leq d$; 
\item for each $1\leq k\leq r_K$, $s_k$ is a constant that depends only on $k$ and $\mathcal{P}$;
\item for each $r_K+1\leq k\leq d$, there are $\#T_k$ values of $s_k$, and $s_k$ can be written as $s_k=b_k+u$, where $b_k$ depends only on $k$ and $\mathcal{P}$ and $u\in T_k$; here $T_k$ is the set defined in Theorem \ref{T:QC-NF-integrals-pts-full-statement}.
\end{itemize}
In total, we get $\#T_{r_K+1}\cdot \#T_{r_K+2} \cdots \#T_{d}$ systems to
solve. In general, we can solve such systems using a multivariable version
of Hensel's Lemma proved in~\cite{Conrad-Multivariable-Hensel}. For a practical algorithm for systems in two variables,
see \cite[Appendix A]{BBBM-QCNF}. We use a {\tt SageMath}-implementation of this algorithm due to Francesca Bianchi.

We mention a specific situation that allows us to conclude that there is a unique solution to the system. This is a special case of \cite[Theorem A.1]{BBBM-QCNF}. Recall that $p\geq 3$. Let us write 
\[
F_{k,j}(z_j)=c_{k,j,1}z_j+c_{k,j,2}z_j^2 + c_{k,j,3}z_j^3 + O(z_j^4).
\]
Denote $\mathcal{M}=(c_{k,j,1})_{1\leq k,j\leq d}\in \Q_p^{d\times d}$. If $\det(\mathcal{M})\not\equiv 0\pmod{p}$, then the system has a unique solution $(z_1,\ldots,z_d)\in (p\Z_p)^d$. This is particularly useful in the special case when $\mathcal{P}\in\sigma(\mathcal{U}(\O_K))$, and we can prove that $z_1=\cdots=z_d=0$ is the unique solution to the system above. This implies that $\mathcal{P}$ is the unique point in the intersection $\mathcal{D}(\mathcal{P})\cap \sigma(\mathcal{U}(\O_K))$.

We have implemented the method described in this section,
see~{\url{https://github.com/steffenmueller/LinQC}}.

\subsubsection{Precision analysis}\label{subsec:precOK}
We only give a brief sketch of the precision analysis in a special case, which
includes our example in~\S\ref{subsec:Example-QC-OK-real-points}. In this
case, the precision analysis is quite
similar to \cref{subsec:prec-qc-for-Z}. As the variables in the
functions 
$\rho_k$ are separated, the precision analysis for the computation of the
$\rho_k$ is the same as over $\Q$. For the same reason, solving for the roots of the
system of power series equations is similar to the
case over $\Q$ and is easier to analyze than the general case, at least when
there are no multiple roots modulo $p^n$. This indeed happens in our example. 
In this case, we use naive Hensel lifting, as explained in \cite[Theorem
A.1, Algorithm 1]{BBBM-QCNF}. In particular, we never reach 
Step (4)(ii) in~\cite[Algorithm 1]{BBBM-QCNF}, which would require a more
involved precision analysis.

\subsection{A real quadratic example}\label{subsec:Example-QC-OK-real-points}
Let $K=\Q(w)$, where $w^2=7$.
Using an implementation of the method discussed
in~\S\ref{subsec:alg-OK-points}, we now show that the set of 
$\O_K$-integral points $\calU(\O_K)$ on the hyperelliptic
curve $X/K$ defined by 
\begin{equation*}\label{rqex}
  y^2 = x^6 + x^5 - wx^4 + (1-w)x^3 + x^2 + wx + w
\end{equation*}
is exactly 
\begin{equation}\label{rqex:intpts}
  \calU(\O_K) = \{(-1,0), (1, \pm 2)\}\,.
\end{equation}
We are not aware of any other method that solves this problem.

It is easy to see that $\#J(K)_{\mathrm{tors}}=1$. 
The {\tt Magma}-implementation of the two-descent algorithm
from~\cite{Sto01} shows that $\rk(J/K)=3$. The points $a_1=[(1,-2)-\infty_-]$ and $a_2=[(1,2)-(-1,0)]$ are
independent and all differences of small $K$-rational points on $X$ lie in
the subgroup $\langle a_1, a_2\rangle \subset J(K)\otimes \Q$ (note 
that $X(K)$ contains the small non-integral points $\infty_{\pm}, 
(\frac{4-w}{3}, \frac{196-94w}{27})$).
Let $a_3\in J(K)$ be given by the Mumford representation
$$
a_3 = \left(x^2 + \frac{125+ 43w}{81}x + \frac{28+197w}{567}, \;
\frac{121513+ 40337w}{45927}x + \frac{85001+ 23797w}{45927}\right)\,.
$$
Then $J(K)\otimes \Q = \langle a_1,a_2,a_3\rangle$. 
We checked that in fact the subgroup of $J(K)$ generated by $a_1,a_2,a_3$
is saturated at all primes below~100.

The prime $3$ splits as $3\O_K =
\p_1\p_2$ and $X$ has good reduction at $\p_1$ and $\p_2$.
For all $1\le j \le 3$, the points $2a_j$ have the property that
$2a_j$ has a representative $D_j\in \Div^0(C)$ such that for both $i\in
\{1,2\}$,  $D_j\otimes K_{\p_i}$ is of the form $Q_1+Q_2-R_1-R_2$, where
$Q_1,Q_2\in X(K_{\p_i})$ reduce to affine points and $R_1,R_2\in X(K)$ 
are the points with $x$-coordinate~1.
Working with these representatives, we easily check that
Condition~\ref{cond4} holds. Hence
Theorem~\ref{T:QC-NF-integrals-pts-full-statement}  is applicable.

We now go through the steps of the number field version of
Algorithm~\ref{alg:QC-Z}, discussed in~\S\ref{subsec:alg-OK-points}.
We choose $Q=(-1,0)$ to be our integral base point and we use
the cyclotomic idèle class character on $K$, which we denote by $\chi$ and
which we drop from the notation.
Recall that for a real quadratic field,  $\chi$ is, up to scalar, the only
continuous id\'ele class character. 
{Since $J$ has ordinary reduction at $\p_1$ and $\p_2$, we may choose
the unit root subspace $W_{\p_i}$ for the local heights $h_{\p_i}$.}
Lemma~\ref{L:away} shows that $T\colonequals T_1=\{0\}$.
Using the above-mentioned representatives, we compute the local heights
$h_q(\infty_--\infty_+, D_j)$ for $q\nmid 3$ via the {\tt Magma} command
{\tt LocalIntersectionData}, which implements the algorithm
from~\cite{Mul14}. We compute the local heights
$h_{\p_i}(\infty_--\infty_+, D_j)$ using the {\tt
SageMath}-implementation of the algorithm from
\cite[\S3.2]{p-adicheightsHC}. 
This allows us to solve for coefficients $\alpha_{j,i}$ (note that these
are not unique). The coefficients $\beta_{j,i}$ are simpler to compute, requiring only
integrals of holomorphic differentials.
In this way, we find power series expansion of the two functions $\rho_1$
and $\rho_2$ in each residue disc. We solve for their common 
zeros $(z_1,z_2)\in X(K\otimes\Q_p)$ using a {\tt SageMath}-implementation
of~\cite[Algorithm~A.1]{BBBM-QCNF}, due to Francesca Bianchi.

In addition to the known integral points $(-1,0), (1, \pm2)$, there are~8
common pairs of zeros.
To show that they are not in the image of
$\sigma(\calU(\O_K))$, we apply the Mordell--Weil sieve. See~\cite{BBM2}
and~\cite[Appendix]{Jen-Netan-QCRP1} for more details. For each extra
zero $\underline{z}$ we compute
$\underline{\tilde{\alpha}}=(\tilde{\alpha}_1,\tilde{\alpha}_2,\tilde{\alpha}_3)\in
(\Z/3^5\Z)^3$, such that if $\underline{z}=\sigma(P)$ for some $P\in X(K)$,
then $[P-Q] = \alpha_1a_1+\alpha_2a_2+\alpha_3a_3$, where $\alpha_i\in \Z$ reduces to
$\tilde{\alpha}_i$.
We use the prime ideals of $\O_K$ generated by $(29, w+6), (37,
w+28)$ and $(2017, w+845)$, respectively. The residue fields $k_v$ of 
these primes $v$ all have the
property that $\#J(k_v)/2\cdot 3^5J(k_v)$ is large compared to $\#C(k_v)$ 
(it was not sufficient to work only modulo a power of~3).

Finally, we run through all triples in $(\Z/2\cdot 3^5\Z)^3$ 
that reduce to one of the $\underline{\tilde{\alpha}}$ and we show that, for at
least one $v$ as above, the corresponding coset of $2\cdot 3^5J(k_v)$ does
not intersect the image of $C(k_v)$.
This finishes the proof of~\eqref{rqex:intpts}.
For the code that we used in this example,
see~{\url{https://github.com/steffenmueller/LinQC}.

\subsection{Obstructions to linear quadratic Chabauty}

Let us now consider hyperelliptic
curves $X/\Q\colon y^2=f(x)=d_{2g+2}x^{2g+2}+\cdots+d_0$ such that $f\in \Z[x]$ is squarefree. 
Assume first that $d_{2g+2}>1$ is squarefree, so that the elementary
approach described in Appendix~\ref{appendix:Runge} is not applicable.
If we denote $K\colonequals \Q(\sqrt{d_{2g+2}})$, then  $\infty_{\pm}\in X(K)\backslash X(\Q)$. Therefore, we cannot use the method from Section~\ref{sec:QC-over-Z} to determine the $\Z$-integral points on $X$. 
It is tempting to try to compute  
$\calU(\O_K)$ using the approach of Section~\ref{subsec:alg-OK-points}, as
this would also suffice to 
find $\calU(\Z)$. Unfortunately, this does not work, as we now explain.

Let $p$ be a prime of good reduction that splits in $K$; write $p\O_K=\p_1\p_2$. Since $K$ is a real quadratic field, there is only one $p$-adic height $h(\cdot,\cdot)$ on $J(K)$ (up to scaling), induced by the cyclotomic idèle class character. 
However, it turns out that the restriction of
\begin{equation*}\label{eq:lambda}
\lambda\colon  J(K)\longrightarrow \Q_p\,;\quad a\mapsto h(\infty_- -\infty_+,a)\,.
\end{equation*}
to $J(\Q)$ is identically zero, due to Galois equivariance
of the Coleman-Gross height. See \cite[\S 6.5.5]{Gajovic-thesis} for more
details, and an alternative proof gives more precise
information on the local heights in \cite[Lemma 6.5.6, Remark
6.5.7]{Gajovic-thesis}. It follows that the resulting locally analytic
function is identically zero and hence cannot be used for quadratic
Chabauty, see~\cite[\S 6.6.1]{Gajovic-thesis}.
In fact, a necessary condition to compute the  $\O_K$-integral points over a
real quadratic field $K$ using our method, is that the curve is not a
quadratic twist of a curve defined over $\Q$.
The problem persists for higher degree number fields containing $K$, and  
similar problems arise if we use inert primes rather than split primes. 

The issue described above generalizes to arbitrary quadratic fields if we
use the cyclotomic $p$-adic height. 
If $K=\Q(\sqrt{d_{2g+2}})$ is an imaginary quadratic field, the same
problem occurs if we instead use the
anticyclotomic idèle class character (see \cite[\S 6.5.4,\S6.6.1]{Gajovic-thesis}). 
Instead, one could use the sum of the cyclotomic and anticyclotomic height
as in \cite{Jen-Netan-QCRP1}. In this
case, we have to require the stronger  rank condition $\rk(J/K)\leq 2g-1$.}

\begin{remark} 
We expect that a modified version of the approach
from \cite{BBM1} using $p$-adic Arakelov theory and double Coleman integrals can be used to compute the $\Z$-integral points on even degree hyperelliptic curves $y^2=f(x)$ when $f\in \Z[x]$, assuming that the usual rank condition is satisfied.
We also expect that such an approach can be generalized to number fields as in~\cite{BBBM-QCNF}.
\end{remark}

\appendix

\section{A special case of Runge's method}\label{appendix:Runge}
For expository purposes, we recall how Runge's method  \cite{Runge-Original-Paper} (see \cite[Chapter 5]{Schoof-Book-Catalan} for a nice survey) can sometimes be used to determine the integral points on monic hyperelliptic curves of even degree in an elementary way. For a more general approach and a proof that uses Runge's machinery, see \cite[Theorem 3]{Walsh-Runge}. Throughout this appendix (except in Remark \ref{rmk:negative-leading-coefficient}), we assume that $f=d_{2n}x^{2n}+\cdots+d_0\in \Z[x]$ is nonconstant and squarefree, and that $d_{2n}$ is a non-zero square, say $d_{2n}=b^2$, where $b\in\mathbb{Z}_{>0}$.

\begin{theorem}\label{T:Runge-bound}
Let $X/\Q$ be the hyperelliptic curve defined by $y^2=f(x)$.
There is an explicit constant $N>0$ depending only on the coefficients of $f$, such that if $(x_0,y_0)\in X(\Q)$ with $x_0,y_0\in \Z$, then $|x_0|\leq N$. 
\end{theorem} 

We prove Theorem \ref{T:Runge-bound} using Lemma~\ref{decomposition} and Lemma~\ref{bound}.

\begin{lemma} \label{decomposition}
There exists a
polynomial 
\begin{equation}\label{E:g}
g(x)=bx^n+c_{n-1}x^{n-1}+\cdots+c_1 x+c_0\in\mathbb{Q}[x]
\end{equation}
such that $\deg(f-g^2)\leq n-1$. Furthermore, we have $$c_i\in\dfrac{1}{(2b)^{2(n-i)-1}}\mathbb{Z}\, \quad \text{ for }\;0\leq i\leq n-1\,.$$
\end{lemma}
\begin{proof}
We explicitly construct a polynomial $g$ satisfying~\eqref{E:g} by equating coefficients
of $x^k$ for $n \leq k\leq 2n-1$ in $g^2$ and $f$. 
From these equalities, we see inductively that each $c_i$ for $n-1\geq i\geq 0$ is uniquely
determined, and that $c_i\in\dfrac{1}{(2b)^{2(n-i)-1}}\mathbb{Z}$.
  Thus, $g\in\dfrac{1}{(2b)^{2n-1}}\mathbb{Z}[x]\subset
\mathbb{Q}[x]$.
\end{proof}

Let us denote $h\colonequals f-g^2$, where $g$ is defined by Lemma~\ref{decomposition}. Multiply $h$ and $f$ by the minimal power $(2b)^{2k}$ of $2b$ such that $(2b)^{2k}h, (2b)^{2k}f, (2b)^{k}g$ have integral coefficients.
By Lemma \ref{decomposition},
  we know that $k\leq 2n-1$. Let $F\colonequals (2b)^{2k} f$, $G\colonequals (2b)^kg$
  and $H\colonequals (2b)^{2k}h$; then $F=G^2+H$. We can write $$G(x)=p_nx^n+\cdots+p_1x+p_0,\hspace*{3mm} H(x)=q_{n-1}x^{n-1}+\cdots+q_1x+q_0,$$
where $p_0,\ldots,p_n,q_0,\ldots,q_{n-1}\in\mathbb{Z}$ such that $p_n\neq 0$, and not all $q_i$'s
are zero. 

The integral points on $X$ inject into the integral points on the curve  
$X'\colon y^2=F(x)$ via
$(x_0,y_0)\in X(\Q)\mapsto (x_0,(2b)^ky_0)\in X'(\Q)$.
We can bound the integral points on
$X'$ in the following way. Denote
\[M \colonequals  \dfrac{\displaystyle\sum_{i=0}^{n-1}(2|p_i|+|q_i|)+1}{2|p_n|}.
\]

\begin{lemma} \label{bound} If $(x_0,y_0)\in X'(\Q)$ is an
integral point, then $H(x_0)=0$ or 
$|x_0|\leq M$.
\end{lemma}
\begin{proof}
We prove that $F(x_0)$ lies
between $(G(x_0)-1)^2$ and $(G(x_0)+1)^2$ for almost all $x_0$. On the one
  hand, we have   \begin{align*}
    |(G(x)\pm1)^2-G(x)^2|&=|2G(x)\pm1|\geq 2|G(x)|-1\\
    &\geq
    2|p_n||x|^n-2|p_{n-1}||x|^{n-1}-\cdots-2|p_1||x|-2|p_0|-1\equalscolon
    P_U(|x|)\,.
  \end{align*}

On the other hand,
$$|H(x)|\leq |q_{n-1}||x|^{n-1}+\cdots+|q_1||x|+|q_0|.$$

Therefore we deduce
$$P_U(|x|)-|H(x)|\geq 2|p_n||x|^n-(2|p_{n-1}|+|q_{n-1}|)|x|^{n-1}-\cdots-(2|p_1|+|q_1|)|x|-(2|p_0|+|q_0|)-1.$$
For $x\in \mathbb{Z}$, $x\neq 0$, we have the inequalities $|x|^k\leq |x|^{n-1}$ for all
$0\leq k\leq n-1$, so
$$\displaystyle\sum_{i=0}^{n-1}(2|p_i|+|q_i|)|x|^i+1\leq
\left(\displaystyle\sum_{i=0}^{n-1}(2|p_i|+|q_i|)+1\right)|x|^{n-1}.$$
For $|x| > M$ it follows that
$$\left(\displaystyle\sum_{i=0}^{n-1}(2|p_i|+|q_i|)+1\right)|x|^{n-1}<2|p_n||x|^n.$$
Thus, for $|x_0|>M$, we have
$$|(G(x_0)\pm1)^2-G(x_0)^2|\geq P_U(|x_0|)>|H(x_0)|,$$
implying that $G(x_0)^2+H(x_0)$ lies between $(G(x_0)-1)^2$ and $(G(x_0)+1)^2$. Hence, if it is a square of an integer, it must be exactly $G(x_0)^2$, implying $H(x_0)=0$. 
\end{proof}

\begin{proof}[Proof of Theorem~\ref{T:Runge-bound}]
  Let $(x_0,y_0)\in X(\Q)$ be an integral point. Using the injection of
  integral points of $X$ into the integral points of $X'$, Lemma
  \ref{bound}, implies that $|x_0|\leq M$ or $H(x_0)=0$. If $H(x_0)=0$, we
  can easily bound $|x_0|$. 
Furthermore, the proof is constructive, and the required bounds are computed only in terms of coefficients of $f$. Hence we have described an algorithm that in a finite number of steps can determine all integral points on $X$, finishing the proof of Theorem \ref{T:Runge-bound}.
\end{proof}

\begin{remark}\hfill
\begin{enumerate}
\item We do not state the bound $N$ explicitly in terms of the coefficients of $f$. However, it is clear from the construction that the bound is a rational function in terms of the coefficients of $f$. In practice, we found that the bound is sufficiently small to easily compute the integral points.
\item The bounds from the algorithm can be improved. Namely, using standard real analysis, we can locate 
all zeros of $H$ and find all real intervals on which the inequality $P_U(|x|)\leq|H(x)|$ is satisfied.
\item We can determine $\overline{X}(\mathbb{F}_p)$, for some set of small primes $p$, imposing conditions on $x$-coordinates modulo $p$. In this way, we can speed up the search for integral points on $X$.
\end{enumerate}
\end{remark}

\begin{remark}\label{rmk:negative-leading-coefficient}
If we assume that $d_{2n}<0$, then there is also an elementary method to determine the integral points on $X$. Since $\displaystyle \lim_{x\to\infty_-}f(x)=\lim_{x\to\infty_+}f(x)=-\infty$, it follows that $f(x)\geq 0$ can only happen for finitely many $x\in \Z$. Using basic calculus, it is easy to determine an explicit constant $M'$, depending only on the coefficients of $f$, such that if $|x|>M'$, then $f(x)\leq -1$.
\end{remark}
\bibliographystyle{alpha}
\bibliography{References}

@article{Tri,
      title={Restriction of Scalars {C}habauty and the ${S}$-unit equation}, 
      author={Nicholas Triantafillou},
      journal={Arxiv preprint, \url{https://arxiv.org/abs/2006.10590}},
      year={2021},
      eprint={2006.10590},
      primaryClass={math.NT}
}

@manual{Sage,
  Key          = {SageMath},
  Author       = {{The Sage Developers}},
  Title        = {{S}ageMath, the {S}age {M}athematics {S}oftware {S}ystem ({V}ersion 9.8)},
  note         = {{\tt https://www.sagemath.org}},
  Year         = {2023},
}

@article {Sto01,
    AUTHOR = {Stoll, Michael},
     TITLE = {Implementing 2-descent for {J}acobians of hyperelliptic
              curves},
   JOURNAL = {Acta Arith.},
  FJOURNAL = {Acta Arithmetica},
    VOLUME = {98},
      YEAR = {2001},
    NUMBER = {3},
     PAGES = {245--277},
      ISSN = {0065-1036},
   MRCLASS = {11G30 (11G10 14H25 14H40)},
  MRNUMBER = {1829626},
MRREVIEWER = {Edward F. Schaefer},
       DOI = {10.4064/aa98-3-4},
       URL = {https://doi.org/10.4064/aa98-3-4},
}

@article{Hol12,
   author={Holmes, David},
   title={Computing {N}\'eron-{T}ate heights of points on hyperelliptic {J}acobians},
   journal={J. Number Theory},
   volume={132},
   year={2012},
   number={6},
   pages={1295--1305},
}

@article{Mul14,
  title={Computing canonical heights using arithmetic intersection theory},
  author={M{\"u}ller, J. Steffen},
   JOURNAL = {Math. Comp.},
  volume={83},
  number={285},
  pages={311--336},
  year={2014}
}

@article {dJH15,
    AUTHOR = {Holmes, David and de Jong, Robin},
     TITLE = {Asymptotics of the {N}\'{e}ron height pairing},
   JOURNAL = {Math. Res. Lett.},
  FJOURNAL = {Mathematical Research Letters},
    VOLUME = {22},
      YEAR = {2015},
    NUMBER = {5},
     PAGES = {1337--1371},
      ISSN = {1073-2780},
   MRCLASS = {14G40 (14D06 14H1o5)},
  MRNUMBER = {3488379},
MRREVIEWER = {Yu Yasufuku},
       DOI = {10.4310/MRL.2015.v22.n5.a5},
       URL = {https://doi.org/10.4310/MRL.2015.v22.n5.a5},
}

@article {Jen-Even-Degree-CI,
    AUTHOR = {Balakrishnan, Jennifer S.},
     TITLE = {Coleman integration for even-degree models of hyperelliptic
              curves},
   JOURNAL = {LMS J. Comput. Math.},
  FJOURNAL = {LMS Journal of Computation and Mathematics},
    VOLUME = {18},
      YEAR = {2015},
    NUMBER = {1},
     PAGES = {258--265},
   MRCLASS = {11S80 (14G22)},
  MRNUMBER = {3349319},
       DOI = {10.1112/S1461157015000029},
       URL = {https://doi.org/10.1112/S1461157015000029},
}

@article {BBHeights,
    AUTHOR = {Balakrishnan, Jennifer S. and Besser, Amnon},
     TITLE = {Computing local {$p$}-adic height pairings on hyperelliptic
              curves},
   JOURNAL = {Int. Math. Res. Not.},
  FJOURNAL = {International Mathematics Research Notices. IMRN},
      YEAR = {2012},
    NUMBER = {11},
     PAGES = {2405--2444},
      ISSN = {1073-7928},
   MRCLASS = {11G50},
  MRNUMBER = {2926986},
MRREVIEWER = {Lenny Fukshansky},
       DOI = {10.1093/imrn/rnr111},
       URL = {https://doi.org/10.1093/imrn/rnr111},
}

@article {BBBM-QCNF,
    AUTHOR = {Balakrishnan, Jennifer S. and Besser, Amnon and Bianchi,
              Francesca and M\"{u}ller, J. Steffen},
     TITLE = {Explicit quadratic {C}habauty over number fields},
   JOURNAL = {Israel J. Math.},
  FJOURNAL = {Israel Journal of Mathematics},
    VOLUME = {243},
      YEAR = {2021},
    NUMBER = {1},
     PAGES = {185--232},
      ISSN = {0021-2172},
   MRCLASS = {11G30 (11S80 11Y50 14H50)},
  MRNUMBER = {4299146},
       DOI = {10.1007/s11856-021-2158-5},
       URL = {https://doi.org/10.1007/s11856-021-2158-5},
}

@article { BBM1,
      author = "Jennifer S. Balakrishnan and Amnon Besser and J. Steffen Müller",
      title = "Quadratic {C}habauty: $p$-adic heights and integral points on hyperelliptic curves",
      journal = "Journal für die reine und angewandte Mathematik",
      year = "01 Nov. 2016",
      publisher = "De Gruyter",
      address = "Berlin, Boston",
      volume = "2016",
      number = "720",
      doi = "https://doi.org/10.1515/crelle-2014-0048",
      pages=      "51 - 79",
      url = "https://www.degruyter.com/view/journals/crll/2016/720/article-p51.xml"
}

@article{BBM2,
author = {Balakrishnan, Jennifer and Besser, Amnon and Müller, J. Steffen},
year = {2017},
pages = {1403--1434},
title = {Computing integral points on hyperelliptic curves using quadratic {C}habauty},
volume = {86},
   JOURNAL = {Math. Comp.},
doi = {10.1090/mcom/3130}
}

@incollection {CCExperimentsGenus3,
    AUTHOR = {Balakrishnan, Jennifer S. and Bianchi, Francesca and
              Cantoral-Farf\'{a}n, Victoria and \c{C}iperiani, Mirela and
              Etropolski, Anastassia},
     TITLE = {Chabauty-{C}oleman experiments for genus 3 hyperelliptic
              curves},
 BOOKTITLE = {Research directions in number theory---{W}omen in {N}umbers
              {IV}},
    SERIES = {Assoc. Women Math. Ser.},
    VOLUME = {19},
     PAGES = {67--90},
 PUBLISHER = {Springer, Cham},
      YEAR = {2019},
      ISBN = {978-3-030-19478-9; 978-3-030-19477-2},
   MRCLASS = {11G30 (11S80 11Y50 14H40)},
  MRNUMBER = {4069379},
MRREVIEWER = {Timo\ Keller},
       DOI = {10.1007/978-3-030-19478-9\_3},
       URL = {https://doi.org/10.1007/978-3-030-19478-9_3},
}

@incollection {BBK,
    AUTHOR = {Balakrishnan, Jennifer S. and Bradshaw, Robert W. and Kedlaya,
              Kiran S.},
     TITLE = {Explicit {C}oleman integration for hyperelliptic curves},
 BOOKTITLE = {Algorithmic number theory},
    SERIES = {Lecture Notes in Comput. Sci.},
    VOLUME = {6197},
     PAGES = {16--31},
 PUBLISHER = {Springer, Berlin},
      YEAR = {2010},
   MRCLASS = {14G22 (11G20)},
  MRNUMBER = {2721410},
MRREVIEWER = {Denis Ibadula},
       DOI = {10.1007/978-3-642-14518-6\_6},
       URL = {https://doi.org/10.1007/978-3-642-14518-6_6},
}

@article {Jen-Netan-QCRP1,
    AUTHOR = {Balakrishnan, Jennifer S. and Dogra, Netan},
     TITLE = {Quadratic {C}habauty and rational points, {I}: {$p$}-adic
              heights},
      NOTE = {With an appendix by J. Steffen M\"{u}ller},
   JOURNAL = {Duke Math. J.},
  FJOURNAL = {Duke Mathematical Journal},
    VOLUME = {167},
      YEAR = {2018},
    NUMBER = {11},
     PAGES = {1981--2038},
      ISSN = {0012-7094},
   MRCLASS = {14G05 (11G50 14G40)},
  MRNUMBER = {3843370},
MRREVIEWER = {Ariyan Javanpeykar},
       DOI = {10.1215/00127094-2018-0013},
       URL = {https://doi.org/10.1215/00127094-2018-0013},
}

@article {BDMTV2,
    AUTHOR = {Balakrishnan, Jennifer S. and Dogra, Netan and M\"{u}ller, J.
              Steffen and Tuitman, Jan and Vonk, Jan},
     TITLE = {Quadratic {C}habauty for modular curves: algorithms and
              examples},
   JOURNAL = {Compos. Math.},
  FJOURNAL = {Compositio Mathematica},
    VOLUME = {159},
      YEAR = {2023},
    NUMBER = {6},
     PAGES = {1111--1152},
      ISSN = {0010-437X},
   MRCLASS = {11G18 (11G50 11Y50 14G05 14G35)},
  MRNUMBER = {4589060},
       DOI = {10.1112/s0010437x23007170},
       URL = {https://doi.org/10.1112/s0010437x23007170},
}

@article {Besser-padic-Arakelov,
    AUTHOR = {Besser, Amnon},
     TITLE = {{$p$}-adic {A}rakelov theory},
   JOURNAL = {J. Number Theory},
  FJOURNAL = {Journal of Number Theory},
    VOLUME = {111},
      YEAR = {2005},
    NUMBER = {2},
     PAGES = {318--371},
      ISSN = {0022-314X},
   MRCLASS = {14G40 (11G25 11G50 14G20 14G25)},
  MRNUMBER = {2130113},
MRREVIEWER = {J\"{o}rg Jahnel},
       DOI = {10.1016/j.jnt.2004.11.010},
       URL = {https://doi.org/10.1016/j.jnt.2004.11.010},
}

@article {Raymond-David-Steffen,
    AUTHOR = {van Bommel, Raymond and Holmes, David and M\"{u}ller, J. Steffen},
     TITLE = {Explicit arithmetic intersection theory and computation of
              {N}\'{e}ron-{T}ate heights},
   JOURNAL = {Math. Comp.},
  FJOURNAL = {Mathematics of Computation},
    VOLUME = {89},
      YEAR = {2020},
    NUMBER = {321},
     PAGES = {395--410},
      ISSN = {0025-5718},
   MRCLASS = {14G40 (11G30 11G50 37P30)},
  MRNUMBER = {4011549},
MRREVIEWER = {Zhizhong Huang},
       DOI = {10.1090/mcom/3441},
       URL = {https://doi.org/10.1090/mcom/3441},
}

@article{Magma,
author = "Bosma, Wieb and Cannon, John and Playoust, Catherine",
journal = "J. Symbolic Comput.",
number = "3--4",
pages = "235--265",
title = "The {M}agma algebra system. {I}. {T}he user language",
volume = "24",
year = "1997"
}

@article{Bruin-Stoll-MW-Sieve, 
title={The {M}ordell–{W}eil sieve: proving non-existence of rational points on curves}, 
volume={13}, 
DOI={10.1112/S1461157009000187}, 
journal={LMS Journal of Computation and Mathematics}, 
publisher={London Mathematical Society}, 
author={Bruin, Nils and Stoll, Michael}, 
year={2010}, 
pages={272–306}
}

@article{Conrad-Multivariable-Hensel,
author = {Conrad, Keith},
title = {A multivariable {H}ensel’s lemma},
journal =
{\url{https://kconrad.math.uconn.edu/blurbs/gradnumthy/multivarhensel.pdf},
            accessed online: 28/07/2023},
}

@incollection {Coleman-Gross-Heights,
    AUTHOR = {Coleman, Robert F. and Gross, Benedict H.},
     TITLE = {{$p$}-adic heights on curves},
 BOOKTITLE = {Algebraic number theory},
    SERIES = {Adv. Stud. Pure Math.},
    VOLUME = {17},
     PAGES = {73--81},
 PUBLISHER = {Academic Press, Boston, MA},
      YEAR = {1989},
   MRCLASS = {11F85 (11G10 11G30)},
  MRNUMBER = {1097610},
MRREVIEWER = {Dipendra Prasad},
       DOI = {10.2969/aspm/01710073},
       URL = {https://doi.org/10.2969/aspm/01710073},
}

@article{Dogra-Unlikely-Intersections-Chabauty-Kim,
Author={Dogra, Netan},
Title={{U}nlikely intersections and the {C}habauty-{K}im method over number fields
},
journal={Mathematische Annalen, to appear},
year  = {2023}
}

@article{p-adicheightsHC,
Author={Gajovi\'c, Stevan and M\"uller, J. Steffen},
Title={Computing $p$-adic heights on
hyperelliptic curves
},
journal={ArXiv preprint, \url{https://arxiv.org/abs/2307.15787}},
  Year = {2023}
}

@phdthesis{Gajovic-thesis,
    AUTHOR = {Gajovi\'c, Stevan},
     TITLE = {Variations on the method of {C}habauty and {C}oleman},
      SCHOOL = {University of Groningen},
      YEAR = {2022},
}

@book {Gouvea-p-adic-numbers,
    AUTHOR = {Gouv\^{e}a, Fernando Q.},
     TITLE = {{$p$}-adic numbers: An introduction},
    SERIES = {Universitext},
   EDITION = {{F}irst},
 PUBLISHER = {Springer-Verlag, Berlin},
      YEAR = {1993},
     PAGES = {vi+282},
      ISBN = {3-540-56844-1},
   MRCLASS = {11S80 (11-01 11S85 12J25 30G06)},
  MRNUMBER = {1251959},
MRREVIEWER = {Daniel Barsky},
       DOI = {10.1007/978-3-662-22278-2},
       URL = {https://doi.org/10.1007/978-3-662-22278-2},
}

@incollection {Gross-local-heights,
    AUTHOR = {Gross, Benedict H.},
     TITLE = {Local heights on curves},
 BOOKTITLE = {Arithmetic geometry ({S}torrs, {C}onn., 1984)},
     PAGES = {327--339},
 PUBLISHER = {Springer, New York},
      YEAR = {1986},
   MRCLASS = {11G30},
  MRNUMBER = {861983},
}

@book {Liu,
    AUTHOR = {Liu, Qing},
     TITLE = {Algebraic geometry and arithmetic curves},
    SERIES = {Oxford Graduate Texts in Mathematics},
    VOLUME = {6},
      NOTE = {Translated from the French by Reinie Ern\'{e},
              Oxford Science Publications},
 PUBLISHER = {Oxford University Press, Oxford},
      YEAR = {2002},
     PAGES = {xvi+576},
      ISBN = {0-19-850284-2},
   MRCLASS = {14-01 (11G30 14A05 14A15 14Gxx 14Hxx)},
  MRNUMBER = {1917232},
MRREVIEWER = {C\'{\i}cero Carvalho},
}

@article {Lorenzini-Tucker,
    AUTHOR = {Lorenzini, Dino and Tucker, Thomas J.},
     TITLE = {Thue equations and the method of {C}habauty-{C}oleman},
   JOURNAL = {Invent. Math.},
  FJOURNAL = {Inventiones Mathematicae},
    VOLUME = {148},
      YEAR = {2002},
    NUMBER = {1},
     PAGES = {47--77},
      ISSN = {0020-9910},
   MRCLASS = {11G30 (11D59 11G10 14G25)},
  MRNUMBER = {1892843},
MRREVIEWER = {Matthew H. Baker},
       DOI = {10.1007/s002220100186},
       URL = {https://doi.org/10.1007/s002220100186},
}

@article{Chabauty-Coleman-MCP12,
author = "McCallum, William and Poonen, Bjorn",
journal = "Explicit methods in number theory, Panor. Synthèses",
pages = "99--117",
publisher = "Soc. Math. France, Paris",
title = "The method of {C}habauty and {C}oleman",
volume = "36",
year = "2012"
}

@article {Runge-Original-Paper,
    AUTHOR = {Runge, Carl},
     TITLE = {Ueber ganzzahlige {L}\"{o}sungen von {G}leichungen zwischen zwei
              {V}er\"{a}nderlichen},
   JOURNAL = {J. Reine Angew. Math.},
  FJOURNAL = {Journal f\"{u}r die Reine und Angewandte Mathematik. [Crelle's
              Journal]},
    VOLUME = {100},
      YEAR = {1887},
     PAGES = {425--435},
      ISSN = {0075-4102},
   MRCLASS = {DML},
  MRNUMBER = {1580107},
       DOI = {10.1515/crll.1887.100.425},
       URL = {https://doi.org/10.1515/crll.1887.100.425},
}

@book {Schoof-Book-Catalan,
    AUTHOR = {Schoof, Ren\'{e}},
     TITLE = {Catalan's conjecture},
    SERIES = {Universitext},
 PUBLISHER = {Springer-Verlag London, Ltd., London},
      YEAR = {2008},
     PAGES = {x+124},
      ISBN = {978-1-84800-184-8},
   MRCLASS = {11D61 (11D41 11R18 11R29)},
  MRNUMBER = {2459823},
MRREVIEWER = {Yann Bugeaud},
}

@article {Siksek-Chabauty-NF,
    AUTHOR = {Siksek, Samir},
     TITLE = {Explicit {C}habauty over number fields},
   JOURNAL = {Algebra Number Theory},
  FJOURNAL = {Algebra \& Number Theory},
    VOLUME = {7},
      YEAR = {2013},
    NUMBER = {4},
     PAGES = {765--793},
      ISSN = {1937-0652},
   MRCLASS = {11G30 (14H25)},
  MRNUMBER = {3095226},
MRREVIEWER = {Joseph H. Silverman},
       DOI = {10.2140/ant.2013.7.765},
       URL = {https://doi.org/10.2140/ant.2013.7.765},
}

@article {Walsh-Runge,
    AUTHOR = {Walsh, P. G.},
     TITLE = {Corrections to: ``{A} quantitative version of {R}unge's
              theorem on {D}iophantine equations'' [{A}cta {A}rith. {\bf 62}
              (1992), no. 2, 157--172; {MR}1183987 (94a:11037)]},
   JOURNAL = {Acta Arith.},
  FJOURNAL = {Acta Arithmetica},
    VOLUME = {73},
      YEAR = {1995},
    NUMBER = {4},
     PAGES = {397--398},
      ISSN = {0065-1036,1730-6264},
   MRCLASS = {11D41 (11D25)},
  MRNUMBER = {1366046},
       DOI = {10.4064/aa-73-4-397-398},
       URL = {https://doi.org/10.4064/aa-73-4-397-398},
}

\end{document}